\documentclass[smallextended,final]{svjour3}
\usepackage[letterpaper,top=2in, bottom=1.5in, left=1in, right=1in]{geometry}

\usepackage{epsfig,epsf,fancybox}
\usepackage{amsmath}
\usepackage{mathrsfs}
\usepackage{amssymb}
\usepackage{amsfonts}
\usepackage{graphicx}
\usepackage{color}
\usepackage[linktocpage,pagebackref,colorlinks,linkcolor=blue,anchorcolor=blue,citecolor=blue,urlcolor=blue,hypertexnames=false]{hyperref}
\usepackage{boxedminipage}
\usepackage{stmaryrd}
\usepackage{multirow}
\usepackage{booktabs}
\usepackage{cite}
\usepackage{float}
\usepackage[ruled,vlined,linesnumbered]{algorithm2e}
\usepackage{algpseudocode}

\usepackage{mathtools}

\usepackage[normalem]{ulem} 

\newcommand{\bm}[1]{\boldsymbol{#1}}

\newcommand{\va}{{\mathbf{a}}}
\newcommand{\vb}{{\mathbf{b}}}
\newcommand{\vc}{{\mathbf{c}}}

\newcommand{\ve}{{\mathbf{e}}}

\newcommand{\vg}{{\mathbf{g}}}

\newcommand{\vv}{{\mathbf{v}}}

\newcommand{\vx}{{\mathbf{x}}}
\newcommand{\vy}{{\mathbf{y}}}
\newcommand{\vz}{{\mathbf{z}}}

\newcommand{\vA}{{\mathbf{A}}}
\newcommand{\vB}{{\mathbf{B}}}

\newcommand{\vI}{{\mathbf{I}}}

\newcommand{\vQ}{{\mathbf{Q}}}

\newcommand{\vtheta}{{\bm{\theta}}}

\newcommand{\cB}{{\mathcal{B}}}

\newcommand{\cE}{{\mathcal{E}}}

\newcommand{\cL}{{\mathcal{L}}}

\newcommand{\vareps}{\varepsilon}

\newcommand{\RR}{\mathbb{R}} 
\newcommand{\vzero}{\mathbf{0}} 

\newcommand{\dist}{\mathrm{dist}}    
\newcommand{\dom}{{\mathrm{dom}}} 

\newcommand{\st}{\mbox{ s.t. }}

\DeclareMathOperator*{\argmin}{arg\,min} 
\DeclareMathOperator*{\argmax}{arg\,max} 

\DeclareMathOperator*{\Min}{minimize}

\newcommand{\bc}{\begin{center}}
\newcommand{\ec}{\end{center}}

\newcommand{\bdm}{\begin{displaymath}}
\newcommand{\edm}{\end{displaymath}}

\newcommand{\beq}{\begin{equation}}
\newcommand{\eeq}{\end{equation}}

\newcommand{\bfl}{\begin{flushleft}}
\newcommand{\efl}{\end{flushleft}}

\newcommand{\bt}{\begin{tabbing}}
\newcommand{\et}{\end{tabbing}}

\newcommand{\beqn}{\begin{eqnarray}}
\newcommand{\eeqn}{\end{eqnarray}}

\newcommand{\beqs}{\begin{align*}} 
\newcommand{\eeqs}{\end{align*}}  

\newtheorem{assumption}{Assumption}

\begin{document}

\title{First-order methods for problems with $O(1)$ functional constraints can have almost the same convergence rate as for unconstrained problems}

\author{Yangyang Xu}

\institute{Y. Xu \at Department of Mathematical Sciences, Rensselaer Polytechnic Institute, Troy, NY 12180\\
\email{xuy21@rpi.edu}}

\date{\today}

\maketitle

\begin{abstract}
First-order methods (FOMs) have recently been applied and analyzed for solving problems with complicated functional constraints. Existing works show that FOMs for functional constrained problems have lower-order convergence rates than those for unconstrained problems. In particular, an FOM for a smooth strongly-convex problem can have linear convergence, while it can only converge sublinearly for a constrained problem if the projection onto the constraint set is prohibited. In this paper, we point out that the slower convergence is caused by the large number of functional constraints but not the constraints themselves. When there are only $m=O(1)$ functional constraints, we show that an FOM can have almost the same convergence rate as that for solving an unconstrained problem, even without the projection onto the feasible set. In addition, given an $\vareps>0$, we show that a complexity result that is better than a lower bound can be obtained, if there are only $m=o(\vareps^{-\frac{1}{2}})$ functional constraints. Our result is surprising but does not contradict to the existing lower complexity bound, because we focus on a specific subclass of problems. Experimental results on quadratically-constrained quadratic programs demonstrate our theory.

\vspace{0.3cm}

\vspace{0.3cm}


\end{abstract}

\section{Introduction}
In this paper, we consider the constrained convex programming
\begin{equation}\label{eq:ccp}
\min_{\vx\in\RR^n} F(\vx):=f(\vx) + h(\vx), \st \vg(\vx):=[g_1(\vx),\ldots, g_m(\vx)]\le\vzero,
\end{equation}
where $f$ is a differentiable strongly-convex function with a Lipschitz continuous gradient, $h$ is a simple closed convex function, and each $g_i$ is convex differentiable and has a Lipschitz continuous gradient. 

For a smooth strongly-convex linearly-constrained problem $\min_\vx\{f(\vx), \st \vA\vx=\vb\}$, \cite{ouyang-xu2021lower-bd} gives a lower complexity bound $O(\frac{1}{\sqrt\vareps})$ of first-order methods (FOMs) to produce an $\vareps$-optimal solution, if $\vA$ can be inquired only by the matrix-vector multiplication $\vA(\cdot)$ and $\vA^\top (\cdot)$. Notice $\{\vx: \vA\vx=\vb\}= \{\vx: \vA\vx\le \vb, -\vA\vx\le -\vb\}$. In addition, if $\nabla f(\vx) + \vA^\top\vy = \vzero$, then  $\nabla f(\vx) + \vA^\top\vy^+ - \vA^\top\vy^- = \vzero$, where $\vy^+\ge\vzero$ and $\vy^-\ge\vzero$ denote the positive and negative parts of $\vy$. Hence, if the linear-equality constrained problem has a KKT point, then so does the equivalent linear-inequality constrained problem. Therefore, the lower bound in \cite{ouyang-xu2021lower-bd} also applies to the inequality constrained problem \eqref{eq:ccp}, if $\vg$ can be accessed only through its function value and derivative. 
However, for the special case of $\vg\equiv \vzero$ or $m=0$, an accelerated proximal gradient method \cite{nesterov2013gradient, lin2015adaptive} can achieve a complexity result $O(\sqrt{\kappa}|\log\vareps|)$ to produce an $\vareps$-optimal solution of \eqref{eq:ccp}, when $f$ is strongly convex. Here, $\kappa$ denotes the condition number.

The worst-case instance constructed in \cite{ouyang-xu2021lower-bd} relies on the condition that $m$ is in the same or higher order of $\frac{1}{\sqrt\vareps}$. For the case with $m=o(\frac{1}{\sqrt\vareps})$, the lower bound $O(\frac{1}{\sqrt\vareps})$ may not hold any more. Examples of \eqref{eq:ccp} with small $m$ include the Neyman-Pearson classification problem \cite{rigollet2011neyman}, fairness-constrained classification \cite{zafar2015fairness}, and the risk-constrained portfolio optimization \cite{gandy2005portfolio}. Therefore, we pose the following question while solving a strongly-convex problem in the form of \eqref{eq:ccp}:
\begin{center}
\fbox{\parbox{0.9\textwidth}{Given $\vareps>0$, can an FOM achieve a better complexity result than $O(\frac{1}{\sqrt\vareps})$ to produce an $\vareps$-optimal solution of \eqref{eq:ccp} when $m=o(\frac{1}{\sqrt\vareps})$, or even achieve $\tilde O(\sqrt\kappa)$ when $m=O(1)$?}
}
\end{center}
Here, an FOM for \eqref{eq:ccp} only uses the function value and derivative information of $f$ and $\vg$ and also the proximal mapping of $h$ and its multiples, and $\tilde O$ suppresses a polynomial of $|\log\vareps|$. We will give an affirmative answer to the above question.

\subsection{Algorithmic framework}
The FOM that we will design and analyze is based on the inexact augmented Lagrangian method (iALM).
The classic AL function of \eqref{eq:ccp} is:
\begin{equation}\label{eq:aug-fun}
\cL_\beta(\vx,\vz) = F(\vx) + \textstyle \frac{\beta}{2}\left\|[\vg(\vx)+\frac{\vz}{\beta}]_+\right\|^2 - \frac{\|\vz\|^2}{2\beta},
\end{equation}
where $\vz$ is the multiplier vector, and $[\va]_+$ takes the compoment-wise positive part of a vector $\va$. The pseudocode of a first-order iALM is shown in Algorithm \ref{alg:ialm}. Notice that $\cL_\beta$ is strongly convex about $\vx$ and concave about $\vz$. Hence, we can directly apply the accelerated proximal gradients in \cite{nesterov2013gradient, lin2015adaptive} to solve each $\vx$-subproblem. However, that way can only give a complexity result of $O(\frac{1}{\sqrt\vareps})$ as shown in \cite{xu2021iter-ialm}, regardless of the value of $m$.
 To have a better overall complexity, we will design a new cutting-plane based FOM to solve each $\vx$-subproblem by utilizing the condition $m=O(1)$ or $m=o(\frac{1}{\sqrt\vareps})$.

\begin{algorithm}[h]
\caption{First-order inexact augmented Lagrangian method for \eqref{eq:ccp}}\label{alg:ialm}
\DontPrintSemicolon
{\small
\textbf{Initialization:} choose $\vx^0, \vz^0$, and $\beta_0>0$\;
\For{$k=0,1,\ldots $}{
Apply a first-order method to find $\vx^{k+1}$ as an approximate solution of $\min_\vx\cL_{\beta_k}(\vx, \vz^k )$.\;
Update $\vz$ by $\vz^{k+1} =[\vz^k+\beta_k \vg(\vx^{k+1})]_+$.\;
Choose $\beta_{k+1}\ge \beta_k$.\;
\If{a stopping condition is satisfied}{
Output $(\vx^{k+1},\vz^{k+1})$ and stop
}
}
}
\end{algorithm}

\subsection{Related works}
We briefly mention some existing works that also study the complexity of FOMs for solving functional constrained problems. 

By using the ordinary Lagrangian function, \cite{nedic2009approximate, nedic2009subgradient} analyze a dual subgradient method for general convex problems. The method needs $O(\vareps^{-2})$ subgradient evaluations to produce an $\vareps$-optimal solution (see the definition in Eq. \eqref{eq:eps-sol} below). For a smooth problem, \cite{necoara2014rate} studies the complexity of an inexact dual gradient (IDG) method. Suppose that an optimal FOM is applied to each outer-subproblem of IDG. Then to produce an $\vareps$-optimal solution, IDG needs $O(\vareps^{-\frac{3}{2}})$ gradient evaluations when the problem is convex, and the result can be improved to $O(\vareps^{-\frac{1}{2}}|\log\vareps|)$ when the problem is strongly convex. For convex problems, the primal-dual FOM proposed in \cite{yu2016primal} achieves an $O(\vareps^{-1})$ complexity result to produce an $\vareps$-optimal solution, and the same-order complexity result has also been established in \cite{xu2020-FOM-AL}. Based on a previous work \cite{lan2016iteration-alm} for affinely constrained problems, \cite{lu2018iteration} gives a modified first-order iALM for solving convex cone programs. The overall complexity of the modified method is $O(\vareps^{-1}|\log\vareps|)$ to produce an $\vareps$-KKT point (see Definition~\ref{def:eps-kkt} below). A similar result has also been shown in \cite{aybat2013augmented} for convex conic programs. A proximal iALM is analyzed in \cite{li2019-piALM}. By a linearly-convergent first-order subroutine for primal subproblems, \cite{li2019-piALM} shows that $O(\vareps^{-1})$ calls to the subroutine are needed for convex problems and  $O(\vareps^{-\frac{1}{2}})$ for strongly convex problems, to achieve either an $\vareps$-optimal or an $\vareps$-KKT point. In terms of function value and derivative evaluations, the complexity result is $O(\vareps^{-1}|\log\vareps|)$ for the convex case and $O(\vareps^{-\frac{1}{2}}|\log\vareps|)$ for the strongly-convex case. Complexity results of FOMs for nonconvex problems with functional constraints have also been established, e.g., \cite{lin2019inexact-PP, li2021rate-improved-ALM, li2020augmented, sahin2019inexact, boob2019stochastic, melo2020iteration, kong2019complexity, cartis2011evaluation}. To produce an $\vareps$-KKT point, the best-known result is $\tilde O(\vareps^{-\frac{5}{2}})$ when the constraints are convex \cite{lin2019inexact-PP, li2021rate-improved-ALM} and $\tilde O(\vareps^{-3})$ when the constraints are nonconvex and satisfy a certain regularity condition \cite{lin2019inexact-PP}.

On solving general nonlinear constrained problems, FOMs have also been proposed under the framework of the level-set method \cite{aravkin2019level, lin2018level-SIOPT, lin2018level-ICML}. For convex problems, the level-set based FOMs can also achieve an $O(\vareps^{-1})$ complexity result to produce an $\vareps$-optimal solution. However, to obtain $\tilde O(\vareps^{-\frac{1}{2}})$, they require strong convexity of both the objective and the constraint functions.

Under the condition of strong duality, \eqref{eq:ccp} can be equivalently formulated as a non-bilinear saddle-point (SP) problem. In this case, one can apply any FOM that is designed for solving non-bilinear SP problems. The work \cite{hamedani2018primal} generalizes the primal-dual method proposed in \cite{chambolle2011first} from the bilinear SP case to the non-bilinear case. If the underlying SP problem is convex-concave, \cite{hamedani2018primal} establishes an $O(\vareps^{-1})$ complexity result to guarantee $\vareps$-duality gap. When the problem is strongly-convex-linear, the result can be improved to $O(\vareps^{-\frac{1}{2}})$. Notice that both results apply to the equivalent ordinary-Lagrangian-based SP problem of \eqref{eq:ccp}. By the smoothing technique, \cite{hien2017inexact} gives an FOM (with both deterministic and stochastic versions) for solving non-bilinear SP problems. To ensure an $\vareps$-duality gap of a strongly-convex-concave problem, the method requires $\tilde O(\vareps^{-\frac{1}{2}})$ primal first-order oracles and $\tilde O(\vareps^{-1})$ dual first-order oracles. While applied to the functional constrained problem \eqref{eq:ccp}, the method in \cite{hien2017inexact} can obtain an $\vareps$-optimal solution by $O(\vareps^{-\frac{1}{2}}|\log\vareps|)$ evaluations on $f$, $\nabla f$, $\vg$, and $J_\vg$.
FOMs for solving the more general variational inequality (VI) problem can also be applied to \eqref{eq:ccp}, such as the mirror-prox method in \cite{nemirovski2004prox}, the hybrid extragradient method in \cite{monteiro2010complexity}, and the accelerated method in \cite{chen2017accelerated}. All of the three methods can have an $O(\vareps^{-1})$ complexity result by assuming smoothness and/or monotonicity of the involved operator.

\subsection{Contributions}
On solving a functional constrained strongly-convex problem, none of the existing works about FOMs (such as those we mentioned previously) could obtain a complexity result better than $\tilde O(\vareps^{-\frac{1}{2}})$. Without specifying the regime of $m$, the task is impossible. We show that when $m=O(1)$ in \eqref{eq:ccp}, an FOM can achieve almost the same-order complexity result (with a difference of at most a polynomial of $|\log\vareps|$) as for solving an unconstrained problem. When $m=o(\vareps^{-\frac{1}{2}})$, we show that a complexity result better than $\tilde O(\vareps^{-\frac{1}{2}})$ can be obtained. The key step in the design of our algorithm is to formulate each primal subproblem into an equivalent SP problem. The SP formulation is strongly concave about the dual variable, and the strong concavity enables the generation of a cutting plane while searching for an approximate dual solution of the SP problem. Since there are $m$ dual variables, we can apply a cutting-plane method to efficiently find an approximate dual solution when $m=O(1)$ or $m=o(\vareps^{-\frac{1}{2}})$. In addition, we extend the idea of a cutting-plane based FOM to the convex and nonconvex cases. For these two cases, we show that an FOM for problems with $O(1)$ functional constraints can also achieve almost the same-order complexity result as for solving unconstrained problems.

\subsection{Assumptions and notation} 

Throughout our analysis for strongly-convex problems, we make the following assumptions.

\begin{assumption}[smoothness]\label{assump:smooth}
$f$ is $L_f$-smooth, i.e., $\nabla f$ is $L_f$-Lipschitz continuous. In addition, each $g_i$ is smooth, and the Jacobian matrix $J_\vg = [\nabla g_1^\top; \ldots; \nabla g_m^\top]$ is $L_g$-Lipschitz continuous. 
\end{assumption}

\begin{assumption}[bounded domain and convexity]\label{assump:cvx}
The domain of $h$ is bounded with a diameter $D_h=\max_{\vx,\vy\in\dom(h)}\|\vx-\vy\|<\infty$. The functions $h$ and $\{g_i\}$ are all convex.
\end{assumption}

The above two assumptions imply the boundedness of $\vg$ and $J_\vg$ on $\dom(h)$. We use $G$ and $B_g$ respectively for their bounds, namely,
\begin{equation}\label{eq:def-G-Bg}
G=\max_{\vx\in\dom(h)} \|\vg(\vx)\|,\quad B_g = \max_{\vx\in\dom(h)} \|J_\vg(\vx)\|.
\end{equation} 

\begin{assumption}[strong convexity]\label{assump:scvx}
The smooth function $f$ is $\mu$-strongly convex with $\mu>0$.
\end{assumption}

\begin{assumption}[strong duality]\label{assump:kkt}
There is a primal-dual solution $(\vx^*,\vz^*)$ satisfying the KKT conditions of \eqref{eq:ccp}, i.e.,
$\vzero\in \partial F(\vx^*) + J_\vg(\vx^*)^\top \vz^*,\, \vz^*\ge \vzero, \quad g(\vx^*)\le\vzero, \quad \vg(\vx^*)^\top \vz^* = 0.$
\end{assumption}

When Assumotion~\ref{assump:kkt} holds, it is easy to have (cf. \cite[Eqn. 2.4]{xu2020primal})
\begin{equation}\label{eq:opt-ineq}
F(\vx)-F(\vx^*) + \langle \vz^*, \vg(\vx)\rangle \ge0,\, \forall\, \vx\in\dom(h).
\end{equation}

\noindent\textbf{Notation.}~~For a real number $a$, we use $\lceil a\rceil$ to denote the smallest integer that is no less than $a$ and $\lceil a\rceil_+$ the smallest nonnegative integer that is no less than $a$. $\cB_\delta(\vx)$ denotes a ball with radius $\delta$ and center $\vx$. If $\vx=\vzero$, we simply use $\cB_\delta$. We define $\cB_\delta^+$ as the intersection of $\cB_\delta$ with the nonnegative orthant, so in the $n$-dimensional space, $\cB_\delta^+=\cB_\delta\cap\RR_+^n$. We use $V_m(\delta)$ for the volume of $\cB_\delta$ in the $m$-dimensional space. $[n]$ denotes the set $\{1,\ldots,n\}$. Given a closed convex set $X\subseteq\RR^n$ and a point $\vx\in \RR^n$, we define $\dist(\vx, X)=\min_{\vy\in X}\|\vy-\vx\|$. We use $O$, $\Theta$, and $o$ with standard meanings, while in the complexity result statement, $\tilde O$ has a similar meaning as $O$ but suppresses a polynomial of $|\log\vareps|$ for a given error tolerance $\vareps>0$. 

\begin{definition}[$\vareps$-KKT point]\label{def:eps-kkt}
Given $\vareps>0$, a point $\bar\vx\in \dom(h)$ is called an $\vareps$-KKT point of \eqref{eq:ccp} if there is $\bar\vz\ge \vzero$ such that
\begin{equation}\label{eq:eps-kkt}
\dist\big(\vzero, \partial_\vx\cL_0(\bar\vx, \bar\vz)\big) \le \vareps, \quad \|[\vg(\bar\vx)]_+\|\le \vareps, \quad \sum_{i=1}^m|\bar z_i g_i(\bar\vx)| \le \vareps,
\end{equation}
where $\cL_0(\vx,\vz)=F(\vx) + \vz^\top \vg(\vx)$ is the ordinary Lagrangian function of \eqref{eq:ccp}.
\end{definition}
By the convexity of $F$ and each $g_i$, and also Assumption~\ref{assump:kkt}, one can easily show that an $\vareps$-KKT point of  \eqref{eq:ccp} must be an $O(\vareps)$-optimal solution, where we call a point $\bar\vx\in\dom(h)$ as an $\vareps$-optimal solution of \eqref{eq:ccp} if
\begin{equation}\label{eq:eps-sol}
\big|F(\bar\vx)-F(\vx^*)\big| \le \vareps,\quad \|[\vg(\bar\vx)]_+\|\le \vareps.
\end{equation}

\subsection{Outline}
The rest of the paper is organized as follows. In section~\ref{sec:apg}, we review an adaptive accelerated proximal gradient method (APG) and give the convergence rate of the iALM. In section~\ref{sec:subroutine}, we design new FOMs (that are better than directly applying the APG) for solving primal subproblems in the iALM. Overall complexity results are shown in section~\ref{sec:overall}. Extensions to convex and nonconvex cases are given in section~\ref{sec:cvx-ncvx}. Numerical experiments are conducted in section~\ref{sec:numerical} to demonstrate our theory, and section~\ref{sec:conclusion} concludes the paper.

\section{An adaptive optimal FOM and convergence rate of iALM}\label{sec:apg}

In this section, we give an adaptive optimal FOM that will be used as a subroutine in our algorithm. Also, we establish the convergence rate of the iALM to produce an approximate KKT point.

\subsection{An adaptive optimal FOM for strongly-convex composite problems}
Consider the problem
\begin{equation}\label{eq:comp-prob}
\Min_{\vx\in\RR^n} P(\vx):=\psi(\vx) + r(\vx),
\end{equation}
where $\psi$ is a differentiable $\mu_\psi$-strongly convex function with $L_\psi$-Lipschitz continuous gradient, and $r$ is a closed convex function. Several optimal FOMs have been given in the literature for solving \eqref{eq:comp-prob}, e.g., in \cite{nesterov2013gradient, lin2015adaptive}. In this paper, we choose the adaptive APG in \cite{lin2015adaptive}, and we rewrite it in Algorithm~\ref{alg:nesterov} with a few modified steps for our purpose to produce near-stationary points.  

\begin{algorithm}[h]
\caption{An adaptive optimal first-order method for \eqref{eq:comp-prob}: $\widehat\vx=\mathrm{APG}(\psi, r, \mu_\psi, L_{\min}, \bar\vareps, \gamma_1,\gamma_2)$}\label{alg:nesterov}
\DontPrintSemicolon
{\small
\textbf{Input:} minimum Lipschitz $L_{\min} >0$, increase rate $\gamma_1>1$, decrease rate $\gamma_2\ge 1$, and error tolerance $\bar\vareps >0$.\;
\textbf{Prestep:} choose any $\widetilde\vy=\vy^0\in\dom(r)$ and let $\widetilde L= L_{\min}/\gamma_1$\;
\Repeat{$\psi(\widetilde\vx) \le \psi(\widetilde\vy)+\langle \nabla \psi(\widetilde\vy), \widetilde\vx - \widetilde\vy\rangle + \frac{\widetilde L}{2}\|\widetilde\vx-\widetilde\vy\|^2$}{
$\widetilde L \gets \gamma_1 \widetilde L$ and let $\widetilde\vx = \argmin_{\vx} \langle \nabla \psi(\widetilde\vy), \vx\rangle + \frac{\widetilde L}{2}\|\vx-\widetilde\vy\|^2+r(\vx)$
}
\textbf{Initialization:} let $\vx^{-1}=\vx^0=\widetilde\vx$, $L_0 = \max\{ L_{\min}, \widetilde L / \gamma_2\}$, and $\alpha_{-1}=1$\;
\For{$k=0,1,\ldots$}{
$\widetilde L \gets L_{k}/\gamma_1$\;
\Repeat{$\psi(\widetilde\vx) \le \psi(\widetilde\vy)+\langle \nabla \psi(\widetilde\vy), \widetilde\vx - \widetilde\vy\rangle + \frac{\widetilde L}{2}\|\widetilde\vx-\widetilde\vy\|^2$}{
$\widetilde L \gets \gamma_1 \widetilde L$, $\alpha_k \gets \sqrt{\mu_\psi/\widetilde L}$, and $\widetilde \vy \gets \vx^k + \frac{\alpha_k(1-\alpha_{k-1})}{\alpha_{k-1}(1+\alpha_{k})}(\vx^k-\vx^{k-1})$\;
let $\widetilde\vx = \argmin_{\vx} \langle \nabla \psi(\widetilde\vy), \vx\rangle + \frac{\widetilde L}{2}\|\vx-\widetilde\vy\|^2+r(\vx)$
}
$\widehat L \gets \widetilde L / \gamma_1$; \;
\Repeat{$\psi(\widehat\vx) \le \psi(\widetilde\vx)+\langle \nabla \psi(\widetilde\vx), \widehat\vx - \widetilde\vx\rangle + \frac{\widehat L}{2}\|\widehat\vx-\widetilde\vx\|^2$}{
increase $\widehat L \gets \gamma_1 \widehat L$;\;
let $\widehat\vx = \argmin_{\vx} \langle \nabla \psi(\widetilde\vx), \vx\rangle + \frac{\widehat L}{2}\|\vx-\widetilde\vx\|^2+r(\vx)$; {\color{blue}\algorithmiccomment{modified step to guarantee near-stationarity at $\widehat\vx$}}
}
set $\vx^{k+1}=\widetilde\vx$, $\widehat\vx^{k+1}=\widehat\vx$, and $L_{k+1}= \max\{L_{\min}, \widetilde L/\gamma_2\}$;\;
\If{$\dist\big(\vzero,\partial P(\widehat\vx)\big) \le \bar\vareps$}{
return $\widehat\vx$ and stop.
}
} 
}
\end{algorithm}

The results in the next theorem are from Theorem~1 of \cite{lin2015adaptive}.

\begin{theorem}\label{thm:nesterov}
The generated sequence $\{\vx^k\}_{k\ge0}$ by Algorithm~\ref{alg:nesterov} satisfies
\begin{equation}\label{eq:rate-nesterov}
P(\vx^{k+1})-P(\vx^*)\le \left(1-\sqrt\frac{\mu_\psi}{\gamma_1 L_\psi}\right)^{k+1} \left(P(\vx^0)-P(\vx^*)+\frac{\mu_\psi}{2}\|\vx^0-\vx^*\|^2\right),\, \forall\, k\ge0,
\end{equation}
where $\vx^*$ is the optimal solution of \eqref{eq:comp-prob}.
\end{theorem}

By the above theorem, we can easily bound the distance of $\widehat\vx^k$ to stationarity for each $k$.
\begin{theorem}\label{thm:nesterov-stationarity}
The generated sequence $\{\widehat\vx^k\}_{k\ge0}$ satisfies
$$\dist\big(\vzero, \partial P(\widehat\vx^{k+1})\big)\le \left(\textstyle\sqrt{\gamma_1 L_\psi} + \frac{L_\psi}{\sqrt{L_{\min}}}\right)\sqrt{2(P(\vx^0)-P(\vx^*))+{\mu_\psi}\|\vx^0-\vx^*\|^2}\left(1-\sqrt\frac{\mu_\psi}{\gamma_1 L_\psi}\right)^{\frac{k+1}{2}},\, \forall\, k\ge0.$$
\end{theorem}

\begin{proof}
First notice that if $\widehat L\ge L_\psi$, it must hold $\psi(\widehat\vx) \le \psi(\widetilde\vx)+\langle \nabla \psi(\widetilde\vx), \widehat\vx - \widetilde\vx\rangle + \frac{\widehat L}{2}\|\widehat\vx-\widetilde\vx\|^2$, and when this inequality holds, we have (cf. \cite[Lemma 2.1]{xu2013block}) $P(\widetilde\vx) - P(\widehat\vx) \ge \frac{\widehat L}{2}\|\widehat\vx - \widetilde\vx\|^2$. Since $P(\widetilde\vx) - P(\widehat\vx) \le P(\widetilde\vx) - P(\vx^*)$, we have $\frac{\widehat L}{2}\|\widehat\vx - \widetilde\vx\|^2\le P(\widetilde\vx) - P(\vx^*)$, which together with the fact $\widehat L\ge L_{\min}$ implies
\begin{equation}\label{bd-xtilde-x}
\textstyle \frac{\widehat L^2}{2}\|\widehat\vx - \widetilde\vx\|^2\le \widehat L \big(P(\widetilde\vx) - P(\vx^*)\big),\quad \|\widehat\vx - \widetilde\vx\|^2\le \frac{2}{L_{\min}}\big(P(\widetilde\vx) - P(\vx^*)\big).
\end{equation}
 In addition, from the optimality condition of $\widehat\vx$, it follows $\vzero\in \nabla \psi(\widetilde\vx) + \widehat L(\widehat\vx -\widetilde\vx) + \partial r(\widehat\vx)$, and thus 
\begin{equation}\label{eq:dist-gradp}
\dist(\vzero, \partial P(\widehat\vx)) \le \|\nabla \psi(\widehat\vx)-\nabla \psi(\widetilde\vx)\| + \widehat L\|\widehat\vx -\widetilde\vx\|\le (L_\psi + \widehat L)\|\widehat\vx -\widetilde\vx\|.
\end{equation}
By \eqref{bd-xtilde-x} and \eqref{eq:dist-gradp}, we have 
$$\dist(\vzero, \partial P(\widehat\vx))\le(L_\psi + \widehat L)\|\widehat\vx -\widetilde\vx\| \le \sqrt{2(P(\widetilde\vx) - P(\vx^*))}\left(\textstyle\sqrt{\widehat L} + \frac{L_\psi}{\sqrt{L_{\min}}}\right).$$
Therefore, the desired result follows from \eqref{eq:rate-nesterov}, the fact $\widehat L\le \gamma_1 L_\psi$, and the above inequality with $\widehat\vx =\widehat\vx^{k+1}$ and $\widetilde\vx = \vx^{k+1}$. 
\end{proof}

From \cite[Theorem~3.1]{FISTA2009}, we have  
\begin{equation}\label{eq:initial-dec}
\textstyle P(\vx^0)-P(\vx^*) \le \frac{\gamma_1 L_\psi\|\vy^0-\vx^*\|^2}{2}.
\end{equation}
Hence, we can obtain the following complexity result by Theorem~\ref{thm:nesterov-stationarity} together with \eqref{eq:initial-dec}. 

\begin{corollary}\label{cor:iter-nesterov}
Assume that $\dom(r)$ is bounded with a diameter $D_r=\max_{\vx_1,\vx_2\in \dom(r)}\|\vx_1-\vx_2\|$. Given $\bar\vareps>0$, $\gamma_1>1$, $\gamma_2\ge1$ and $L_{\min}>0$, Algorithm~\ref{alg:nesterov} needs at most $T$ evaluations on the objective value of $\psi$ and the gradient $\nabla \psi$ to produce $\widehat\vx$ such that $\dist(\vzero,\partial P(\widehat\vx))\le \bar\vareps$, where
$$T=\left(1+\lceil {\textstyle\log_{\gamma_1}\frac{L_\psi}{L_{\min}}}\rceil_+\right)\left(1 + 2 \left\lceil \textstyle 2\sqrt{\frac{\gamma_1 L_\psi}{\mu_\psi}}\log\left(\frac{D_r}{\bar\vareps}\big(\sqrt{\gamma_1 L_\psi} + \frac{L_\psi}{\sqrt{L_{\min}}}\big)\sqrt{2\gamma_1 L_\psi+\mu_\psi}\right)\right\rceil_+\right).$$
\end{corollary}

\begin{proof}
Since $\dom(r)$ has a diameter $D_r$, we have from Theorem~\ref{thm:nesterov-stationarity} and \eqref{eq:initial-dec} that
$$\dist\big(\vzero, \partial P(\widehat\vx^{k+1})\big)\le D_r\left(\textstyle\sqrt{\gamma_1 L_\psi} + \frac{L_\psi}{\sqrt{L_{\min}}}\right)\sqrt{{2\gamma_1 L_\psi+\mu_\psi}}\left(1-\sqrt\frac{\mu_\psi}{\gamma_1 L_\psi}\right)^{\frac{k+1}{2}},\, \forall\, k\ge0.$$ 
Hence, if $k+1\ge K$, then $\dist\big(\vzero, \partial P(\widehat\vx^{k+1})\big)\le\bar\vareps$, where
$$\textstyle K = \left\lceil \frac{2\log\left(\frac{D_r}{\bar\vareps}\big(\sqrt{\gamma_1 L_\psi} + \frac{L_\psi}{\sqrt{L_{\min}}}\big)\sqrt{2\gamma_1 L_\psi+\mu_\psi}\right)}{\log (1-\sqrt\frac{\mu_\psi}{\gamma_1 L_\psi})^{-1}}\right\rceil_+,$$
namely, after at most $K$ iterations, the algorithm will produce a point $\widehat\vx$ satisfying $\dist(\vzero,\partial P(\widehat\vx))\le \bar\vareps$.

Notice that the conditions in Lines~5, 11, and 17 of Algorithm~\ref{alg:nesterov} will hold if $\widetilde L\ge L_\psi$ and $\widehat L\ge L_\psi$. Hence, every iteration will evaluate the objective value of $\psi$ and the gradient $\nabla \psi$ at most $2(1+\lceil\log_{\gamma_1}\frac{L_\psi}{L_{\min}}\rceil_+)$ times. Now using the fact $\log(1-a)^{-1}\ge a,\,\forall\, 0<a < 1$, we obtain the desired result by also counting the objective and gradient evaluations to obtain $\vx^0$.
\end{proof}

\subsection{Convergence rate of iALM}

The next lemma is from Eq. (3.20) and the proof of Lemma~7 of \cite{xu2021iter-ialm}.
\begin{lemma}
Let $\{(\vx^{k},\vz^k)\}$ be generated from Algorithm~\ref{alg:ialm} with $\vz^0=\vzero$. Suppose 
\begin{equation}\label{eq:obj-err-ialm}
\cL_{\beta_k}(\vx^{k+1},\vz^k) \le \min_\vx \cL_{\beta_k}(\vx,\vz^k) + e_k,\, \forall\, k=0,1,\ldots,
\end{equation}
for an error sequence $\{e_k\}$. Then
\begin{equation}\label{eq:bd-z-k}
\textstyle \|\vz^k\|^2 \le 4\|\vz^*\|^2 + 4\sum_{t=0}^{k-1}\beta_t e_t,\text{ and }~ \|\vz^k\| \le 2\|\vz^*\| + \sqrt{2\sum_{t=0}^{k-1}\beta_t e_t},\, \forall\, k\ge 1.
\end{equation}
\end{lemma}

By this lemma and also the strong convexity of $F$, we can show the following result.
\begin{lemma}
Let $\{(\vx^{k},\vz^k)\}$ be generated from Algorithm~\ref{alg:ialm} with $\vz^0=\vzero$. If $\dist\big(\vzero, \partial_\vx \cL_{\beta_k}(\vx^{k+1},\vz^k)\big)\le \vareps_k,\, \forall\, k\ge0$ for a sequence $\{\vareps_k\}$, then
\begin{equation}\label{eq:bd-z-k-kkt}
\textstyle \|\vz^k\|^2 \le 4\|\vz^*\|^2 + 4\sum_{t=0}^{k-1}\beta_t \frac{\vareps_t^2}{\mu},\text{ and }~ \|\vz^k\| \le 2\|\vz^*\| + \sqrt{2\sum_{t=0}^{k-1}\beta_t \frac{\vareps_t^2}{\mu}},\, \forall\, k\ge 1.
\end{equation}
\end{lemma}

\begin{proof}
Let $\vx_*^{k+1}$ be the minimizer of $\cL_{\beta_k}(\vx,\vz^k)$ about $\vx$. Then $\vzero\in \partial_\vx \cL_{\beta_k}(\vx_*^{k+1},\vz^k)$. Also, it follows from $\dist\big(\vzero, \partial_\vx \cL_{\beta_k}(\vx^{k+1},\vz^k)\big)\le \vareps_k$ that there is $\vv\in \partial_\vx \cL_{\beta_k}(\vx^{k+1},\vz^k)$ and $\|\vv\|\le \vareps_k$. Since $F$ is $\mu$-strongly convex, $\cL_{\beta_k}(\vx,\vz^k)$ is also $\mu$-strongly convex about $\vx$. Then we have $\langle \vv, \vx^{k+1}-\vx_*^{k+1}\rangle\ge\mu\|\vx^{k+1}-\vx_*^{k+1}\|^2$, which together with the Cauchy-Schwarz inequality gives $\|\vx^{k+1}-\vx_*^{k+1}\|\le \frac{\|\vv\|}{\mu}\le \frac{\vareps_k}{\mu}$. Now by the convexity of $\cL_{\beta_k}(\cdot,\vz^k)$, it holds
$$\textstyle\cL_{\beta_k}(\vx^{k+1},\vz^k)- \cL_{\beta_k}(\vx_*^{k+1},\vz^k) \le \langle \vv, \vx^{k+1}-\vx_*^{k+1}\rangle\le \frac{\vareps_k^2}{\mu},$$
and thus we have that \eqref{eq:obj-err-ialm} holds with $e_t=\frac{\vareps_t^2}{\mu}$ . Therefore, \eqref{eq:bd-z-k-kkt} follows from \eqref{eq:bd-z-k}.
\end{proof}

\begin{theorem}[convergence rate of iALM]\label{thm:rate-ialm}
Let $\{(\vx^{k},\vz^k)\}$ be generated from Algorithm~\ref{alg:ialm} with $\vz^0=\vzero$. Suppose $\beta_k=\beta_0\sigma^k,\,\forall\,k\ge0$ for some $\sigma>1$ and $\beta_0>0$, and $\dist\big(\vzero, \partial_\vx \cL_{\beta_k}(\vx^{k+1},\vz^k)\big)\le \bar\vareps,\, \forall\, k\ge0$ for a positive number $\bar\vareps$. Then
\begin{align}
\textstyle \big\|[\vg(\vx^{k+1})]_+\big\| \le \frac{4\|\vz^*\|}{\beta_0\sigma^k} + \frac{\bar\vareps(\sqrt\sigma+1)\sqrt{\frac{2}{\mu(\sigma-1)}} }{\sqrt{\beta_0\sigma^k}}\label{eq:feas-ialm},\\
\textstyle \sum_{i=1}^m \left|z_i^{k+1}g_i(\vx^{k+1})\right| \le \frac{9\|\vz^*\|^2}{2\beta_0\sigma^k} + \frac{\bar\vareps^2(8\sigma+1)}{2\mu(\sigma-1)}.\label{eq:cp-ialm}
\end{align}
\end{theorem}

\begin{proof}
From the update of $\vz$, it follows that $g_i(\vx^{k+1}) \le \frac{z_i^{k+1}-z_i^k}{\beta_k}$ for each $i\in [m]$, and thus by \eqref{eq:bd-z-k-kkt}, we have 
$$\textstyle \big\|[\vg(\vx^{k+1})]_+\big\| \le \frac{\|\vz^{k+1}-\vz^k\|}{\beta_k}\le \frac{\|\vz^{k+1}\|+ \|\vz^k\|}{\beta_k} \le \frac{4\|\vz^*\| + \sqrt{2\sum_{t=0}^{k-1}\beta_t \frac{\vareps_t^2}{\mu}}+ \sqrt{2\sum_{t=0}^{k}\beta_t \frac{\vareps_t^2}{\mu}}}{\beta_k}.$$
Plugging into the above inequality $\vareps_t=\bar\vareps,\,\forall\, t\ge0$ and $\beta_k=\beta_0\sigma^k$, we obtain the inequality in \eqref{eq:feas-ialm}.

Furthermore, for each $i\in[m]$, we have 
$$\textstyle \left|z_i^{k+1}g_i(\vx^{k+1})\right| \le \frac{1}{\beta_k}\left|z_i^{k+1}(z_i^{k+1}-z_i^{k})\right|\le \frac{1}{\beta_k}\left((z_i^{k+1})^2 + \frac{(z_i^{k})^2}{8}\right),$$
and thus $\sum_{i=1}^m \left|z_i^{k+1}g_i(\vx^{k+1})\right| \le \frac{1}{\beta_k}\left(\|\vz^{k+1}\|^2 + \frac{\|\vz^k\|^2}{8}\right)$. Now we obtain the result in \eqref{eq:cp-ialm} by plugging the first inequality in \eqref{eq:bd-z-k-kkt}. 
\end{proof}

We make a few remarks here. Given $\vareps>0$, choose $\bar\vareps>0$ such that $\frac{\bar\vareps^2(8\sigma+1)}{2\mu(\sigma-1)}<\vareps$ in Theorem~\ref{thm:rate-ialm}. Notice that $\partial_\vx \cL_{\beta_k}(\vx^{k+1},\vz^k)= \partial_\vx \cL_{0}(\vx^{k+1},\vz^{k+1})$. Hence, from \eqref{eq:feas-ialm} and \eqref{eq:cp-ialm}, it follows that to ensure $\vx^{k+1}$ to be an $\vareps$-KKT point, we need $\beta_0\sigma^k = \Theta(\frac{1}{\vareps})$ and solve $k=\Theta\big(\log_\sigma\frac{1}{\beta_0\vareps}\big)$ $\vx$-subproblems. Since the smooth part of $\cL_{\beta_k}(\cdot, \vz^k)$ has $\Theta(\beta_k)$-Lipschitz continuous gradient, it needs $O(\sqrt{\frac{\beta_k}{\mu}})$ proximal gradient steps if we directly apply Algorithm~\ref{alg:nesterov}. This way, we can guarantee an $\vareps$-KKT point with a total complexity $O(\sqrt\frac{\kappa}{\vareps}|\log \vareps|)$, where $\kappa$ denotes the condition number in some sense. This complexity result has been established in a few existing works, e.g., \cite{lu2018iteration, li2019-piALM}. It is worse by an order of $\sqrt{\frac{1}{\vareps}}$ than the complexity result in Corollary~\ref{cor:iter-nesterov} for the unconstrained case. Generally, we cannot improve it any more because the result matches with the lower bound given in \cite{ouyang-xu2021lower-bd}. 

In the rest of the paper, we show that in some special cases, a better complexity can be obtained. When $m=O(1)$, we show that we can achieve a complexity result $O(\sqrt{\kappa}|\log \vareps|^3)$, which is in almost the same order as the optimal result for the unconstrained case. For a general $m$, we can achieve $O(m\sqrt{\kappa}|\log \vareps|^3)$, which is better than $O(\sqrt\frac{\kappa}{\vareps}|\log \vareps|)$ in the regime of $m=o(\sqrt\frac{1}{\vareps})$.

\section{Better first-order methods for $\vx$-subproblems}\label{sec:subroutine}
When $m$ is small in \eqref{eq:ccp}, we do not directly apply Algorithm~\ref{alg:nesterov} to solve the $\vx$-subproblem $\min_\vx \cL_{\beta_k}(\vx,\vz^k)$ in Algorithm~\ref{alg:ialm}. Instead, we design new and better FOMs that use Algorithm~\ref{alg:nesterov} as a subroutine in the framework of a cutting-plane method. Our key idea is to reformulate the $\vx$-subproblem into a strongly-convex-strongly-concave saddle-point problem, which has a unique primal-dual solution. For the saddle-point formulation, we first find a sufficient-accurate dual solution by a cutting-plane based FOM. Then we find a sufficient-accurate primal solution based on the obtained approximate dual solution.

Below, we give more precise description on how to design better FOMs. Given $\vz\ge\vzero$, let 
$$\textstyle \vtheta(\vx)=\vg(\vx)+\frac{\vz}{\beta}.$$ 
From \eqref{eq:def-G-Bg} and the Mean-Value Theorem, it follows that $\vtheta$ is $B_g$-Lipschitz continuous, namely,
\begin{equation}\label{eq:lip-theta}
\|\vtheta(\vx_1)-\vtheta(\vx_2)\|\le B_g\|\vx_1-\vx_2\|,\, \forall\, \vx_1,\vx_2.
\end{equation}
With $\vtheta$, we can rewrite the problem $\min_\vx \cL_{\beta}(\vx,\vz)$ into
\begin{equation}\label{eq:u-form-x}
\Min_{\vx\in\RR^n} \phi(\vx):=\textstyle F(\vx) + \frac{\beta}{2}\|[\vtheta(\vx)]_+\|^2.
\end{equation}
Notice that $\frac{1}{2}\|[\vtheta(\vx)]_+\|^2 = \max_{\vy\ge\vzero} \left\{\vy^\top \vtheta(\vx) - \frac{1}{2}\|\vy\|^2\right\}$ and $\vy=[\vtheta(\vx)]_+$ reaches the maximum. We re-write \eqref{eq:u-form-x} into
\begin{equation}\label{eq:equiv-sp}
\min_{\vx\in\RR^n} \max_{\vy\ge \vzero}~ \Phi(\vx,\vy):=\textstyle F(\vx) + \beta \left(\vy^\top \vtheta(\vx) - \frac{1}{2}\|\vy\|^2\right).
\end{equation}
Define 
\begin{equation}\label{eq:dual-y}
d(\vy)= \min_{\vx\in\RR^n} \Phi(\vx,\vy),\ \text{ and }\ \bar\vy = \argmax_{\vy\ge\vzero} d(\vy).
\end{equation}
Notice that $d$ is $\beta$-strongly concave, so $\bar\vy$ is the unique maximizer of $d$. 
Also, for a given $\vy\ge\vzero$, define $\vx(\vy)$ as the unique minimizer of $\Phi(\cdot,\vy)$, i.e.,
\begin{equation}\label{eq:def-x-y}
\vx(\vy)=\argmin_{\vx} \Phi(\vx,\vy).
\end{equation} 

In our algorithm design, we first find an approximate solution $\widehat\vy$ of $\max_{\vy\ge\vzero} d(\vy)$ and then find an approximate solution $\widehat\vx$ of $\min_\vx \Phi(\vx,\widehat\vy)$. By controlling the approximation errors, we can guarantee $\widehat\vx$ to be a near-stationary point of $\phi$. On finding $\widehat\vy$, we use a cutting-plane method. Since $d$ is strongly concave, a cutting plane can be generated at a query point $\vy\ge\vzero$, though we can only have an estimate of $\nabla d(\vy)$ by approximately solving $\min_\vx \Phi(\vx,\vy)$. It is unclear whether the same idea works if we directly play with the augmented (or ordinary) Lagrangian dual function because it is not strongly concave.

\subsection{Preparatory lemmas}
We first establish a few lemmas. The next lemma indicates that the complexity of solving $\min_\vx\Phi(\vx,\vy)$ by the APG can be independent of $\beta$, if $\|\vy\|$ is in the same order of $\|\bar\vy\|$. This fact is the key for us to design a better FOM for solving ALM subproblems.

\begin{lemma}\label{lem:bound-feas-barx}
Suppose $\bar\vx$ is the minimizer of $\phi$ in \eqref{eq:u-form-x}. Then $\bar\vy= [\vtheta(\bar\vx)]_+$ is the solution of $\max_{\vy\ge \vzero}d(\vy)$, and $(\bar\vx,\bar\vy)$ is the saddle point of $\Phi$. In addition, let $(\vx^*,\vz^*)$ be the point in Assumption~\ref{assump:kkt}. Then
\begin{equation}\label{eq:bd-theta-xbar}
\textstyle \|\bar\vy\|=\|[\vtheta(\bar\vx)]_+\|\le \frac{2\|\vz^*\|+\|\vz\|}{\beta}.
\end{equation}
\end{lemma}

\begin{proof}
It is easy to see that $\bar\vy= [\vtheta(\bar\vx)]_+$ is the solution of $\max_{\vy\ge \vzero}d(\vy)$ and $(\bar\vx,\bar\vy)$ is a saddle point of $\Phi$; cf. \cite[Corollary 37.3.2]{rockafellar1970convex}. We only need to show \eqref{eq:bd-theta-xbar}. Since $\bar\vx$ is the minimizer of $\phi$, it holds
$$\textstyle F(\bar\vx)+\frac{\beta}{2}\|[\vtheta(\bar\vx)]_+\|^2 \le F(\vx^*)+ \frac{\beta}{2}\|[\vtheta(\vx^*)]_+\|^2 = F(\vx^*)+ \frac{\beta}{2}\left\|\big[{\textstyle\vg(\vx^*)+\frac{\vz}{\beta}}\big]_+\right\|^2 \le F(\vx^*)+\frac{\|\vz\|^2}{2\beta},$$
where the last inequality holds because $\vg(\vx^*)\le \vzero$ and $\vz\ge\vzero$.
By the above inequality and \eqref{eq:opt-ineq}, we have 
\begin{align*}
\textstyle \frac{\beta}{2}\|[\vtheta(\bar\vx)]_+\|^2 \le \frac{\|\vz\|^2}{2\beta} + \langle \vz^*, \vg(\bar\vx)\rangle \le \frac{\|\vz\|^2}{2\beta} + \langle \vz^*, \vtheta(\bar\vx)\rangle \le \frac{\|\vz\|^2}{2\beta} + \| \vz^*\|\cdot \|[\vtheta(\bar\vx)]_+\|, 
\end{align*}
which implies the inequality in \eqref{eq:bd-theta-xbar}.
\end{proof}

\begin{lemma}
For any $\vy\ge\vzero$, it holds that
\begin{equation}\label{eq:grad-d}
\nabla d(\vy) = \beta\big(\vtheta(\vx(\vy)) - \vy\big),
\end{equation}
where $\vx(\vy)$ is defined in \eqref{eq:def-x-y}. 
In addition,
\begin{equation}\label{eq:monot-theta}
\beta\big\langle \vy_1-\vy_2, \vtheta(\vx(\vy_1)) - \vtheta(\vx(\vy_2))\big\rangle \le -\mu\|\vx(\vy_1)-\vx(\vy_2)\|^2 ,\ \forall\, \vy_1,\vy_2\ge\vzero,
\end{equation}
and
\begin{equation}\label{eq:lip-x}
\textstyle \|\vx(\vy_1)-\vx(\vy_2)\|\le \frac{\beta B_g}{\mu}\|\vy_1-\vy_2\|,\ \forall\, \vy_1,\vy_2\ge\vzero.
\end{equation}
\end{lemma}

\begin{proof}
The result in \eqref{eq:grad-d} follows from the Danskin Theorem (cf. \cite{bertsekas1999nonlinear}). We only need to show \eqref{eq:monot-theta} and \eqref{eq:lip-x}.

For $i=1,2$, denote $\vx_i=\vx(\vy_i)$. From the definition of $\vx(\vy)$ and the $\mu$-strong convexity of $F$, it holds
\begin{align*}
F(\vx_1)+\beta \vy_1^\top \vtheta(\vx_1) \le F(\vx_2)+\beta \vy_1^\top \vtheta(\vx_2) - \frac{\mu}{2}\|\vx_1-\vx_2\|^2,\\
 F(\vx_2)+\beta \vy_2^\top \vtheta(\vx_2) \le F(\vx_1)+\beta \vy_2^\top \vtheta(\vx_1)- \frac{\mu}{2}\|\vx_1-\vx_2\|^2.
\end{align*}
Adding the above two inequalities gives the result in \eqref{eq:monot-theta}. Now using the $B_g$-Lipschitz continuity of $\vtheta$, we have \eqref{eq:lip-x} from \eqref{eq:monot-theta} and complete the proof.
%
\end{proof}

\begin{lemma}[approximate dual gradient]\label{lem:approx-dgrad}
Given $\widehat\vy\ge\vzero$ and $\delta\ge 0$, let $\widehat\vx$ be an approximate minimizer of $\Phi(\cdot,\widehat\vy)$ such that $\dist\big(\vzero, \partial_\vx \Phi(\widehat\vx,\widehat\vy)\big) \le \delta.$ Then 
$$\textstyle \|\vtheta(\widehat\vx)-\vtheta(\vx(\widehat\vy))\|\le B_g \frac{\delta}{\mu},\ \left\|\beta\big(\vtheta(\widehat\vx) - \widehat\vy\big) - \nabla d(\widehat\vy)\right\| \le \beta B_g \frac{\delta}{\mu}.$$
\end{lemma}
\begin{proof}
From the $\mu$-strong convexity of $F$,  it follows that for each $\vy\ge\vzero$, $\Phi(\cdot, \vy)$ is $\mu$-strongly convex, and thus
${\mu}\|\widehat\vx-\vx(\widehat\vy)\| \le \dist\big(\vzero, \partial_\vx \Phi(\widehat\vx,\widehat\vy)\big) \le \delta,$
which gives $\|\widehat\vx-\vx(\widehat\vy)\|\le \frac{\delta}{\mu}$. Hence, by the $B_g$-Lipschitz continuity of $\vtheta$, we have $\|\vtheta(\widehat\vx)-\vtheta(\vx(\widehat\vy))\|\le B_g \frac{\delta}{\mu}$, and thus from \eqref{eq:grad-d},
$$\textstyle \left\|\beta\big(\vtheta(\widehat\vx) - \widehat\vy\big) - \nabla d(\widehat\vy)\right\| =\beta\|\vtheta(\widehat\vx)-\vtheta(\vx(\widehat\vy))\| \le \beta B_g \frac{\delta}{\mu}.$$
This completes the proof.
\end{proof}

\begin{lemma}\label{lem:bd-grad-phi}
Given $\widehat\vy\ge\vzero$, it holds 
$$\dist\big(\vzero, \partial \phi(\widehat\vx)\big) \le \dist\big(\vzero, \partial_\vx\Phi(\widehat\vx,\widehat\vy)\big) + \beta\|J_\vtheta(\widehat\vx)\|\cdot \|[\vtheta(\widehat\vx)]_+ -\widehat\vy\|,\,\forall \,\widehat\vx\in\dom(h).$$
\end{lemma}

\begin{proof}
It is easy to have $\partial \phi(\widehat\vx)= \partial_\vx\Phi(\widehat\vx,\widehat\vy)+ \beta J_\vtheta^\top(\widehat\vx)([\vtheta(\widehat\vx)]_+-\widehat\vy)$. The desired result now follows from the triangle inequality and the Cauchy-Schwarz inequality.
\end{proof}

\begin{lemma}\label{lem:approx-x-grad}
Given $\bar\vareps>0$, if $\widehat\vy\ge \vzero$ is an approximate solution of $\max_{\vy\ge\vzero} d(\vy)$ such that $\|[\vtheta(\vx(\widehat\vy))]_+ -\widehat\vy\| \le \frac{\bar\vareps}{3\beta B_g}$, and $\widehat\vx$ is an approximate minimizer of $\Phi(\cdot, \widehat\vy)$ such that $\dist\big(\vzero, \partial_\vx\Phi(\widehat\vx,\widehat\vy)\big)\le \frac{\bar\vareps}{3}\min\{1,\, \frac{\mu}{\beta B_g^2}\}$, then $\dist\big(\vzero, \partial \phi(\widehat\vx)\big) \le \bar\vareps$.
\end{lemma}

\begin{proof}
Since $\dist\big(\vzero, \partial_\vx\Phi(\widehat\vx,\widehat\vy)\big)\le \frac{\bar\vareps \mu}{3\beta B_g^2}$, we use Lemma~\ref{lem:approx-dgrad}  with $\delta= \frac{\bar\vareps \mu}{3\beta B_g^2}$ to have $\|\vtheta(\widehat\vx)-\vtheta(\vx(\widehat\vy))\|\le \frac{\bar\vareps}{3\beta B_g}$. In addition, from the nonexpansiveness of $[\cdot]_+$, it follows that $\|[\vtheta(\widehat\vx)]_+-[\vtheta(\vx(\widehat\vy))]_+\|\le\frac{\bar\vareps}{3\beta B_g}$.  Because $\|[\vtheta(\vx(\widehat\vy))]_+ -\widehat\vy\| \le \frac{\bar\vareps}{3\beta B_g}$, we have from the triangle inequality that $\|[\vtheta(\widehat\vx)]_+ -\widehat\vy\|\le \frac{2\bar\vareps}{3\beta B_g}$. The desired result now follows from Lemma~\ref{lem:bd-grad-phi} and $\|J_\vg(\vx)\|\le B_g,\,\forall \,\vx\in\dom(h)$.
\end{proof}

\subsection{the case with a single constraint}
For simplicity, we start with the case of $m=1$, so the bold letters $\vy, \vtheta$ are actually scalars in this subsection.
We show the complexity to produce a point $\widehat\vx$ satisfying $\dist\big(\vzero,\partial\phi(\widehat\vx)\big) \le \bar\vareps$ for a specified error tolerance $\bar\vareps>0$. 
By Lemma~\ref{lem:approx-x-grad}, we can first find a $\widehat \vy\ge \vzero$ such that $|[\vtheta(\vx(\widehat \vy))]_+-\widehat \vy| \le \frac{\bar\vareps}{3\beta B_g}$ and then approximately solve $\min_\vx \Phi(\vx,\widehat\vy)$ to obtain $\widehat\vx$.

Our idea of finding a desired approximate solution $\widehat\vy$ is to first obtain an interval that contains the solution $\bar\vy=\argmax_{\vy\ge 0} d(\vy)$ and then to apply a bisection method. The following lemma shows that for a given $\widehat\vy\ge0$, we can either check if it is a desired approximate solution or obtain the sign of $\nabla d(\widehat\vy)$ so that we know the search direction to have a desired solution.

\begin{lemma}\label{lem:cut-plane-1}
Given $\delta>0$ and $\widehat\vy\ge\vzero$, let $\widehat\vx\in\dom(h)$ be a point satisfying $\dist\big(\vzero,\partial_\vx\Phi(\widehat\vx,\widehat\vy)\big) \le \frac{\mu\delta}{4 B_g}$. If $\big|[\vtheta(\widehat\vx)]_+-\widehat\vy\big|\le \frac{3\delta}{4}$, then $|[\vtheta(\vx(\widehat \vy))]_+-\widehat \vy| \le \delta$. Otherwise, $|[\vtheta(\vx(\widehat \vy))]_+-\widehat \vy| > \frac\delta 2$, and $\nabla d(\widehat\vy)(\vtheta(\widehat\vx) -\widehat\vy) > 0$. 
\end{lemma}

\begin{proof}
From Lemma~\ref{lem:approx-dgrad} and the condition on $\widehat\vx$, it follows that 
\begin{equation}\label{eq:cut-plane-1-ineq1}
\textstyle \big|\vtheta(\widehat\vx)-\vtheta(\vx(\widehat\vy))\big|\le \frac{\delta}{4},\text{ and }\big|\beta\big(\vtheta(\widehat\vx) - \widehat\vy\big) - \nabla d(\widehat\vy)\big| \le  \frac{\beta\delta}{4}.
\end{equation} 
Hence, by the nonexpansiveness of $[\cdot]_+$, it holds $|[\vtheta(\widehat\vx)]_+-[\vtheta(\vx(\widehat\vy))]_+|\le \frac{\delta}{4}$. Then,
by the triangle inequality, we have $|[\vtheta(\vx(\widehat \vy))]_+-\widehat \vy| \le \delta$ if $|[\vtheta(\widehat\vx)]_+-\widehat\vy|\le \frac{3\delta}{4}$ and $|[\vtheta(\vx(\widehat \vy))]_+-\widehat \vy| > \frac\delta 2$ otherwise. 

When $|[\vtheta(\widehat\vx)]_+-\widehat\vy| > \frac{3\delta}{4}$, it must hold $|\vtheta(\widehat\vx)-\widehat\vy| > \frac{3\delta}{4}$ because $\widehat\vy\ge0$, and thus $|\beta(\vtheta(\widehat\vx)-\widehat\vy)| > \frac{3\beta\delta}{4}$.  Therefore, from the second inequality in \eqref{eq:cut-plane-1-ineq1}, we conclude that $\nabla d(\widehat\vy)$ must have the same sign as $\vtheta(\widehat\vx) -\widehat\vy$, because otherwise $\big|\beta\big(\vtheta(\widehat\vx) - \widehat\vy\big) - \nabla d(\widehat\vy)\big|\ge |\beta(\vtheta(\widehat\vx)-\widehat\vy)|> \frac{3\beta\delta}{4}$. This completes the proof.
\end{proof}

By this lemma, we design an interval search algorithm that can either return a point $\widehat\vy\ge0$ such that $|[\vtheta(\vx(\widehat \vy))]_+-\widehat \vy| \le \delta$ or return an interval $Y=[a,b]\subseteq [0,\infty)$ that contains the solution $\bar\vy$. 
The pseudocode is shown in Algorithm~\ref{alg:intv}.

\begin{algorithm}[h]
\caption{Interval search: $Y=\mathrm{IntV}(\beta, \vz, \delta, L_{\min}, \gamma_1,\gamma_2)$}\label{alg:intv}
\DontPrintSemicolon
{\small
\textbf{Input:} multiplier vector $\vz\ge\vzero$, penalty $\beta>0$, target accuracy $\delta>0$, $L_{\min}>0$, and $\gamma_1>1, \gamma_2\ge 1$\;
\textbf{Overhead:} define $\vtheta(\vx) =\vg(\vx)+\frac{\vz}{\beta}$, $\Phi(\vx,\vy)$ as in \eqref{eq:equiv-sp}, and $\bar\vareps= \frac{\mu\delta}{4 B_g}$.\; 
\textbf{Initial step:} call Alg.~\ref{alg:nesterov}: $\widehat\vx = \mathrm{APG}(\psi, h, \mu, L_{\min}, \bar\vareps, \gamma_1,\gamma_2)$ with $\psi = \Phi(\cdot, 0) - h$. {\color{blue}\algorithmiccomment{so $\dist\big(\vzero,\partial_\vx\Phi(\widehat\vx,0)\big) \le \frac{\mu\delta}{4 B_g}$}}\;
\If{$[\vtheta(\widehat\vx)]_+\le \frac{3\delta}{4}$}{
Return $Y=\{0\}$ and stop. {\color{blue}\algorithmiccomment{otherwise, $\nabla d(0)$ is positive}}
}
Let $a=0$, $b = \frac{1}{\beta}$ and call Alg.~\ref{alg:nesterov}: $\widehat\vx = \mathrm{APG}(\psi, h, \mu, L_{\min}, \bar\vareps, \gamma_1,\gamma_2)$ with $\psi = \Phi(\cdot, b) - h$. {\color{blue}\algorithmiccomment{set $b=O(\frac{1}{\beta})$}}\;
\While{$\|[\vtheta(\widehat\vx)]_+-b\| > \frac{3\delta}{4}$ and $\vtheta(\widehat\vx) -b > 0$}{
let $a\gets b$, and increase $b \gets 2b$. {\color{blue}\algorithmiccomment{fine to multiply $b$ by a constant $\sigma > 1$}}\;
call Alg.~\ref{alg:nesterov}: $\widehat\vx = \mathrm{APG}(\psi, h, \mu, L_{\min}, \bar\vareps, \gamma_1,\gamma_2)$ with $\psi = \Phi(\cdot, b) - h$.
}
\If{$\|[\vtheta(\widehat\vx)]_+-b\| \le \frac{3\delta}{4}$}{
Return $Y=\{b\}$ and stop. {\color{blue}\algorithmiccomment{found $\widehat\vy=b$ such that $|[\vtheta(\vx(\widehat \vy))]_+-\widehat \vy| \le \delta$}}
}
\Else{
Return $Y=[a, b]$ and stop. {\color{blue}\algorithmiccomment{found an interval containing $\bar\vy$}}
}
}
\end{algorithm}

Once the stopping condition in Line 4 or 10 is satisfied, then by Lemma~\ref{lem:cut-plane-1}, we immediately obtain a desired $\widehat\vy$ such that $|[\vtheta(\vx(\widehat \vy))]_+-\widehat \vy| \le \delta$. 
The next lemma shows that the algorithm must exist the while loop within a finitely many iterations.

\begin{lemma}\label{lem:finite-stop}
Given $\delta>0$, if $b\ge \frac{2\|\vz^*\|+\|\vz\|}{\beta}$ and $\dist\big(\vzero,\partial_\vx\Phi(\widehat\vx,b)\big) \le \frac{\mu\delta}{4 B_g}$, then either $\|[\vtheta(\widehat\vx)]_+-b\| \le \frac{3\delta}{4}$ or $\vtheta(\widehat\vx) -b < 0$.
\end{lemma}

\begin{proof}
From Lemma~\ref{lem:bound-feas-barx}, it follows that $\bar\vy=[\vtheta(\vx(\bar\vy))]_+ \le \frac{2\|\vz^*\| + \|\vz\|}{\beta}$. The result in \eqref{eq:monot-theta} indicates the decreasing monotonicity of $\vtheta(\vx(\vy))$ with respect to $\vy$. Hence, if $b\ge \frac{2\|\vz^*\|+ \|\vz\|}{\beta}$, then $\vtheta(\vx(b)) \le \vtheta(\vx(\bar\vy)) \le \frac{2\|\vz^*\| + \|\vz\|}{\beta} \le b$, and thus $\vtheta(\vx(b)) -b \le 0$. Now if $|[\vtheta(\widehat\vx)]_+-b| > \frac{3\delta}{4}$, we know from Lemma~\ref{lem:cut-plane-1} that $\nabla d(b) \big(\vtheta(\widehat\vx) -b\big) >0$, and thus $\vtheta(\widehat\vx) -b <0$ since $\nabla d(b)=\beta(\vtheta(\vx(b)) -b) \le 0$. This completes the proof.
\end{proof}

When Algorithm~\ref{alg:intv} exits the while loop, it can output a single point or an interval. The lemma below shows that if an interval is returned, then it will contain the solution $\bar\vy$.

\begin{lemma}\label{lem: return-intv}
Given $\delta>0$, let $Y$ be the return from Algorithm~\ref{alg:intv}. If $Y$ contains a single point $\widehat\vy$, then $|[\vtheta(\vx(\widehat \vy))]_+-\widehat \vy| \le \delta$. Otherwise, $Y$ is an interval $[a,b]$, and it holds that $\nabla d(a) > 0, \nabla d(b) <0$, and $\bar\vy\in [a,b]$. 
\end{lemma}

\begin{proof}
If $Y$ contains a single point $\widehat\vy$, then the condition in either Line 4 or 10 of Algorithm~\ref{alg:intv} is satisfied, and we immediately have $|[\vtheta(\vx(\widehat \vy))]_+-\widehat \vy| \le \delta$ from Lemma~\ref{lem:cut-plane-1}. 

Now suppose that $Y$ is an interval $[a,b]$. From Lemma~\ref{lem:cut-plane-1} and the setting in Line~8 of Algorithm~\ref{alg:intv}, we always have $\nabla d(a) >0$. When the algorithm exits the while loop and returns an interval, we have $\|[\vtheta(\widehat\vx)]_+-b\| > \frac{3\delta}{4}$ but $\vtheta(\widehat\vx) -b \le 0$. Then it follows from Lemma~\ref{lem:cut-plane-1} that $\nabla d(b) < 0$. Therefore, the unique solution $\bar\vy$ must lie in $(a,b)$ by the Mean-Value Theorem and the strong concavity of $d$.
\end{proof}

\begin{remark}\label{rm:finite-stop}
Suppose Algorithm~\ref{alg:intv} returns an interval $[a,b]$. Then Lemma~\ref{lem:finite-stop} indicates that $b\le \frac{1}{\beta}\max\{1, 4\|\vz^*\|+ 2\|\vz\|\}$, and in addition, at most $T+2$ calls are made to Alg.~\ref{alg:nesterov}, where $T$ is the smallest non-negative integer such that $2^T \ge 2\|\vz^*\|+ \|\vz\|$.
\end{remark}

Suppose Algorithm~\ref{alg:intv} returns an interval $[a,b]$. We can then use the bisection method to obtain a desired point $\widehat\vy$. The pseudocode is given in Algorithm~\ref{alg:bisec}.

\begin{algorithm}[h]
\caption{Bisection method for $\max_{\vy\ge 0}d(\vy)$: $(\widehat\vx,\widehat\vy)=\mathrm{BiSec}(\beta, \vz, \delta, L_{\min}, \gamma_1,\gamma_2)$}\label{alg:bisec}
\DontPrintSemicolon
{\small
\textbf{Input:} multiplier vector $\vz\ge\vzero$, penalty $\beta>0$, target accuracy $\delta>0$, $L_{\min}>0$, and $\gamma_1>1, \gamma_2\ge 1$\;
\textbf{Overhead:} define $\vtheta(\vx) =\vg(\vx)+\frac{\vz}{\beta}$, $\Phi(\vx,\vy)$ as in \eqref{eq:equiv-sp}, and $\bar\vareps= \frac{\mu\delta}{4 B_g}$.\; 
Call Alg.~\ref{alg:intv}: $Y=\mathrm{IntV}(\beta, \vz, \delta, L_{\min}, \gamma_1,\gamma_2)$ and denote it as $[a,b]$.{\color{blue}\algorithmiccomment{If $Y$ is a singleton, then $a=b$}}\; 
\While{$b-a> \frac{\mu\delta}{\mu+\beta B_g^2}$}{
let $c=\frac{a+b}{2}$ and call Alg.~\ref{alg:nesterov}: $\widehat\vx = \mathrm{APG}(\psi, h, \mu, L_{\min}, \bar\vareps, \gamma_1,\gamma_2)$ with $\psi = \Phi(\cdot, c) - h$\;
\If{$|[\vtheta(\widehat\vx)]_+-c| \le \frac{3\delta}{4}$}{
Let $\widehat\vy = c$, return $(\widehat\vx,\widehat\vy)$, and stop
}
\ElseIf{$\vtheta(\widehat\vx) -c > 0$}{
let $a\gets c$
}
\Else{
let $b\gets c$.
}
}
Let $\widehat\vy = \frac{a+b}{2}$ and $\widehat\vx = \mathrm{APG}(\psi, h, \mu, L_{\min}, \bar\vareps, \gamma_1,\gamma_2)$ with $\psi = \Phi(\cdot, \widehat\vy) - h$, return $(\widehat\vx,\widehat\vy)$, and stop.
}
\end{algorithm}

By Lemma~\ref{lem:cut-plane-1} and the lemma below, it holds that the returned point $\widehat\vy$ from Algorithm~\ref{alg:bisec} must satisfy $|[\vtheta(\vx(\widehat \vy))]_+-\widehat \vy| \le \delta$.

\begin{lemma}\label{lem:small-intv}
Let $Y=[a,b]\subseteq (0,\infty)$. If $\nabla d(a) >0$, $\nabla d(b) <0$, and $b-a \le \frac{\mu\delta}{\mu+\beta B_g^2}$ for a positive $\delta$, then $|[\vtheta(\vx(\widehat \vy))]_+-\widehat \vy| \le \delta$ for any $\widehat\vy\in [a,b]$.
\end{lemma}

\begin{proof}
Recall from Lemma~\ref{lem:bound-feas-barx} that $\bar\vy=[\vtheta(\vx(\bar \vy))]_+ $. Hence, for any $\widehat\vy\in [a,b]$, we have 
\begin{align}\label{eq:near-opt-haty}
\|[\vtheta(\vx(\widehat \vy))]_+-\widehat \vy\| = &~ \|[\vtheta(\vx(\widehat \vy))]_+-\widehat \vy - [\vtheta(\vx(\bar \vy))]_+ +\bar\vy\|\cr
\le & ~ \|[\vtheta(\vx(\widehat \vy))]_+- [\vtheta(\vx(\bar \vy))]_+\| + \|\widehat\vy-\bar\vy\|\cr
\le & ~ \|\vtheta(\vx(\widehat \vy))- \vtheta(\vx(\bar \vy))\| + \|\widehat\vy-\bar\vy\|\cr
\le & ~ B_g\|\vx(\widehat \vy)- \vx(\bar \vy)\| + \|\widehat\vy-\bar\vy\|\cr
\le &~  \textstyle \frac{\beta B_g^2}{\mu} \|\widehat\vy-\bar\vy\| + \|\widehat\vy-\bar\vy\|,
\end{align}
where we have used the non-expansiveness of $[\cdot]_+$ in the second inequality, the third inequality follows from \eqref{eq:lip-theta}, and the last inequality holds because of \eqref{eq:lip-x}. Now since $\bar\vy\in [a,b]$, we have $\|\widehat\vy-\bar\vy\|\le b-a\le \frac{\mu\delta}{\mu+\beta B_g^2}$, and thus the desired result follows.
\end{proof}


\begin{remark}\label{rm:num-bis}
Since the bisection method halves the interval every time, it takes at most $\lceil\log_2 \frac{(b-a)(\mu+\beta B_g^2)}{\mu\delta}\rceil_+$ halves to reduce an initial interval $[a,b]$ to one with length no larger than $\frac{\mu\delta}{\mu+\beta B_g^2}$. Notice $a\ge 0$ and $b\le \frac{1}{\beta}\max\{1, 4\|\vz^*\|+2 \|\vz\|\}$ from Remark~\ref{rm:finite-stop}. Hence, after $Y$ is obtained, Algorithm~\ref{alg:bisec} will call Algorithm~\ref{alg:nesterov} at most $\left\lceil\log_2 \textstyle \frac{\max\big\{1,\ 4\|\vz^*\|+2 \|\vz\|\big\}(\mu+\beta B_g^2)}{\beta\mu\delta}\right\rceil_+ +1 $ times.
\end{remark}

Below we establish the complexity result of Algorithm~\ref{alg:bisec} to return $\widehat\vy$.

\begin{theorem}[Iteration complexity of BiSec]\label{thm:iter-bisec}
Under Assumptions~\ref{assump:smooth}--\ref{assump:kkt}, Algorithm~\ref{alg:bisec} 
needs at most $T$ evaluations on $f$, $\vtheta$, $\nabla f$, and $J_\vtheta$ to output $\widehat\vx$ and $\widehat\vy\ge 0$ that satisfy $\dist\big(\vzero, \partial_\vx \Phi(\widehat\vx, \widehat\vy)\big)\le \bar\vareps$ and $|[\vtheta(\vx(\widehat \vy))]_+-\widehat \vy| \le \delta$, where $\bar\vareps=\frac{\mu\delta}{4 B_g}$, and
$$T= K\left(1+\lceil {\textstyle\log_{\gamma_1}\frac{L_\vz}{L_{\min}}}\rceil_+\right)\left(1+ 2\left\lceil \textstyle 2\sqrt{\frac{\gamma_1 L_\vz}{\mu}} \log\left(\frac{D_h}{\bar\vareps}\left(\sqrt{\gamma_1 L_\vz} + \frac{L_\vz}{\sqrt{L_{\min}}}\right)\sqrt{2\gamma_1 L_\vz+\mu}\right)\right\rceil_+\right),$$
with  $L_\vz = L_f+ L_g \max\{1, 4\|\vz^*\|+2 \|\vz\|\}$ and
\begin{equation}\label{eq:K-form}
K=3+\left\lceil\log_2(2\|\vz^*\|+ \|\vz\|)\right\rceil_+ + \left\lceil\log_2 \textstyle \frac{\max\big\{1,\ 4\|\vz^*\|+2 \|\vz\|\big\}(\mu+\beta B_g^2)}{\beta\mu\delta}\right\rceil_+.
\end{equation}
\end{theorem}

\begin{proof}
By Remarks~\ref{rm:finite-stop} and \ref{rm:num-bis}, Algorithm~\ref{alg:bisec} calls Algorithm~\ref{alg:nesterov} at most $K$ times, where $K$ is given in \eqref{eq:K-form}. Notice that the gradient of $\psi = \Phi(\cdot, b)-h$ is Lipschitz continuous with constant $L_f+\beta b L_g$. Since $b\le\frac{1}{\beta}\max\{1, 4\|\vz^*\|+2 \|\vz\|\}$ from Remark~\ref{rm:finite-stop}, we apply Corollary~\ref{cor:iter-nesterov} to obtain the desired result.
\end{proof}

\subsection{the case with multiple constraints}


In this subsection, we consider the case of $m>1$. Similar to the case of $m=1$, we use a cutting-plane method to approximately solve $\max_{\vy\ge\vzero} d(\vy)$. The next lemma is the key. It provides the foundation to generate a cutting plane if a query point is not sufficiently close to the solution $\bar\vy$.

\begin{lemma}\label{lem:cut-plane-multiple}
Let $b>0$, and suppose $\|\bar\vy\|\le b$. Given $\delta>0$ and $\widehat\vy\ge\vzero$, let $\widehat\vx\in\dom(h)$ be a point satisfying $\dist\big(\vzero,\partial_\vx\Phi(\widehat\vx,\widehat\vy)\big) \le \min\{\frac{\mu\delta}{4 B_g}, \frac{\mu^2\delta}{8 B_g(\mu+\beta B_g^2)}\}$. If $\|[\vtheta(\widehat\vx)]_+-\widehat\vy\|\le \frac{3\delta}{4}$, then $\|[\vtheta(\vx(\widehat \vy))]_+-\widehat \vy\| \le \delta$. Otherwise, $\big\|[\vtheta(\vx(\widehat \vy))]_+-\widehat \vy\big\| > \frac\delta 2$, and also $\langle \vtheta(\widehat\vx)-\widehat\vy, \vy - \widehat\vy\rangle \ge 0$ for any $\vy\in \cB_\eta(\bar\vy)\cap \cB_b^+$, where $\eta = \min\{b, \eta_+\}$, and $\eta_+$ is the positive root of the equation
\begin{equation}\label{eq:cond-eta}
\textstyle\frac{\mu+\beta B_g^2}{\mu}\left(\eta+ \sqrt{\frac{2\eta B_d}{\beta}}\right)=\frac{\delta}{4}, ~\text{ with }~ B_d =\max_{\vy\in \cB_b^+}\nabla d(\vy).
\end{equation}
\end{lemma}

\begin{proof}
By the same arguments in the proof of Lemma~\ref{lem:cut-plane-1}, we can show that $\|[\vtheta(\vx(\widehat \vy))]_+-\widehat \vy\| \le \delta$ if $\|[\vtheta(\widehat\vx)]_+-\widehat\vy\|\le \frac{3\delta}{4}$ and $\|[\vtheta(\vx(\widehat \vy))]_+-\widehat \vy\| > \frac\delta 2$ otherwise. Hence, we only need to show $\langle \vtheta(\widehat\vx)-\widehat\vy, \vy - \widehat\vy\rangle \ge 0$ for any $\vy\in \cB_\eta(\bar\vy)\cap \cB_b^+$ in the latter case, and we prove this by contradiction.


Suppose $\|[\vtheta(\widehat\vx)]_+-\widehat\vy\| > \frac{3\delta}{4}$ and the following condition holds
\begin{equation}\label{eq:assum-cond}
\langle \vtheta(\widehat\vx)-\widehat\vy, \vy - \widehat\vy\rangle < 0,\text{ for some }\vy\in \cB_\eta(\bar\vy)\cap \cB_b^+.
\end{equation}
By the $\beta$-strong concavity of $d$, it holds 
\begin{equation}\label{eq:ineq-str-d}
d(\vy) \le d(\widehat\vy) + \langle \nabla d(\widehat\vy), \vy-\widehat\vy\rangle - \frac{\beta}{2}\|\vy-\widehat\vy\|^2.
\end{equation}
From the Mean-Value Theorem, it follows that there is $\widetilde \vy$ between $\vy$ and $\bar\vy$ such that $d(\vy)-d(\bar\vy) = \langle \nabla d(\widetilde\vy), \vy -\bar\vy\rangle \ge -\eta B_d $, where the inequality holds because $\vy\in \cB_\eta(\bar\vy)$ and $\widetilde \vy$ must fall in $\cB_b^+$. Since $d(\bar\vy)\ge d(\widehat\vy)$, we have $d(\widehat\vy) - d(\vy) \le d(\bar\vy) - d(\vy)\le \eta B_d$. 
Hence, \eqref{eq:assum-cond} and \eqref{eq:ineq-str-d} imply
\begin{equation}\label{eq:temp-01}
\textstyle\frac{\beta}{2}\|\vy-\widehat\vy\|^2 \le \eta B_d + \langle \beta(\vtheta(\widehat\vx)-\widehat\vy)-\nabla d(\widehat\vy), \widehat\vy-\vy\rangle.
\end{equation}
From Lemma~\ref{lem:approx-dgrad} and the condition $\dist\big(\vzero,\partial_\vx\Phi(\widehat\vx,\widehat\vy)\big) \le \frac{\mu^2\delta}{8B_g(\mu +\beta B_g^2)}$, it follows $\|\beta(\vtheta(\widehat\vx)-\widehat\vy)-\nabla d(\widehat\vy)\|\le \frac{\beta\mu\delta}{8(\mu+\beta B_g^2)}$, which together with \eqref{eq:temp-01} and the Cauchy-Schwartz inequality gives
$$
\textstyle \frac{\beta}{2}\|\vy-\widehat\vy\|^2 \le \eta B_d + \frac{\beta\mu\delta}{8(\mu+\beta B_g^2)} \|\widehat\vy-\vy\|.$$
Solving the above inequality, we have $\|\vy-\widehat\vy\| \le \sqrt{\frac{2\eta B_d}{\beta}} + \frac{\mu\delta}{4(\mu+\beta B_g^2)}$, and since $\|\vy-\bar\vy\|\le \eta$, it holds $\|\bar\vy-\widehat\vy\| \le \eta+ \sqrt{\frac{2\eta B_d}{\beta}} + \frac{\mu\delta}{4(\mu+\beta B_g^2)}$. Now using \eqref{eq:near-opt-haty}, we have
\begin{equation}\label{eq:lbd}
\textstyle\|[\vtheta(\vx(\widehat\vy))]_+-\widehat\vy\| \le \frac{\mu+\beta B_g^2}{\mu}\left(\eta+ \sqrt{\frac{2\eta B_d}{\beta}} + \frac{\mu\delta}{4(\mu+\beta B_g^2)}\right) = \frac{\mu+\beta B_g^2}{\mu}\left(\eta+ \sqrt{\frac{2\eta B_d}{\beta}}\right)+\frac{\delta}{4} \le \frac{\delta}{2},
\end{equation}
where the last inequality follows from the choice of $\eta$.

However, we know that when $\|[\vtheta(\widehat\vx)]_+-\widehat\vy\| > \frac{3\delta}{4}$, it holds $\|[\vtheta(\vx(\widehat \vy))]_+-\widehat \vy\| > \frac\delta 2$, and \eqref{eq:lbd} contradicts to this fact. Therefore, the assumption in \eqref{eq:assum-cond} cannot hold. 
%
This completes the proof.
\end{proof}

Suppose $\|\bar\vy\|\le b$ for some $b>0$. For a given $\widehat\vy\ge\vzero$, let $\widehat\vx$ satisfy the condition required in Lemma~\ref{lem:cut-plane-multiple}. Then if $\|[\vtheta(\widehat\vx)]_+-\widehat\vy\| > \frac{3\delta}{4}$, we find a half-space containing the set $\cB_\eta(\bar\vy)\cap \cB_b^+$, whose volume is at least $4^{-m}V_m(\eta)$ if $\eta\le b$. Therefore, we can apply a cutting-plane method to find a near-optimal $\widehat\vy$. For simplicity, we use the ellipsoid method. The pseudocode is shown in Algorithm~\ref{alg:ellipsoid}. In general, the ellipsoid method is numerically inefficient for high-dimensional problems. However, it can converge fast for solving the low-dimensional dual problem $\min_{\vy\ge\vzero}d(\vy)$, as we will show in the numerical experiments.

\begin{algorithm}[h]
\caption{Ellipsoid Method for $\max_{\vy\ge\vzero} d(\vy)$: $(\widehat\vx,\widehat\vy,\mathrm{FLAG})=\mathrm{Ellipsoid}(\beta,\vz,\delta, b, L_{\min}, \gamma_1,\gamma_2)$}\label{alg:ellipsoid}
\DontPrintSemicolon
{\small
\textbf{Input:} multiplier vector $\vz\ge\vzero$, penalty $\beta>0$, target accuracy $\delta>0$, $b>0$, $L_{\min}>0$, and $\gamma_1>1, \gamma_2\ge 1$\;
\textbf{Overhead:} define $\vtheta(\vx) =\vg(\vx)+\frac{\vz}{\beta}$, $\Phi(\vx,\vy)$ as in \eqref{eq:equiv-sp}, $\bar\vareps=\min\{\frac{\mu\delta}{4 B_g}, \frac{\mu^2\delta}{8B_g(\mu +\beta B_g^2)}\}$, and $\mathrm{FLAG}=0$.\;
Let $\eta_+$ be the positive root of \eqref{eq:cond-eta} and $\eta\gets \min\{b, \eta_+\}$, and set $k=0$.\;
Set $\cE_0=\{\vy\in\RR^m: (\vy-\widehat\vy)^\top \vB^{-1} (\vy-\widehat\vy) \le 1\}$ with $\vB = b^2\vI$ and $\widehat\vy = \vzero$ {\color{blue}\algorithmiccomment{initial ellipsoid}}\;
\While{the volume of $\cE_k > 4^{-m}V_m(\eta)$}{
\If{$\widehat\vy\not\ge\vzero$}{
Let $\va = -\ve_{i_0}$ where $i_0=\argmin_{i\in [m]} \widehat y_i$ {\color{blue}\algorithmiccomment{add a cutting plane $y_{i_0}\ge \widehat y_{i_0}$}}\;
Set $\cE_{k+1}=\{\vy\in\RR^m: (\vy-\widehat\vy)^\top \vB^{-1} (\vy-\widehat\vy) \le 1\}$ with updated $\vB$ and $\widehat\vy$ by 
\begin{equation}\label{eq:update-B-y}
\vB\gets \frac{m^2}{m^2-1}\left(\vB-\frac{2}{(m+1)\va^\top \vB \va} \vB\va(\vB\va)^\top\right),\quad \widehat\vy \gets \widehat\vy - \frac{1}{m+1}\frac{\vB\va}{\sqrt{\va^\top \vB \va}}
\end{equation}  
}
\ElseIf{$\|\widehat\vy\| > b$}{
Let $\va = \widehat\vy$ {\color{blue}\algorithmiccomment{add a cutting plane $\langle \widehat\vy, \vy-\widehat\vy\rangle \le 0$}}\;
Set $\cE_{k+1}=\{\vy\in\RR^m: (\vy-\widehat\vy)^\top \vB^{-1} (\vy-\widehat\vy) \le 1\}$ with $\vB$ and $\widehat\vy$ updated by \eqref{eq:update-B-y}
}
\Else{
Call Alg.~\ref{alg:nesterov}: $\widehat\vx = \mathrm{APG}(\psi, h, \mu, L_{\min}, \bar\vareps, \gamma_1,\gamma_2)$ with $\psi = \Phi(\cdot, \widehat\vy) - h$\;
\If{$\|[\vtheta(\widehat\vx)]_+-\widehat\vy\|\le \frac{3\delta}{4}$}{
$\mathrm{FLAG}=1$, return $(\widehat\vx,\widehat\vy,\mathrm{FLAG})$, and stop {\color{blue}\algorithmiccomment{found $\widehat\vy$ such that $|[\vtheta(\vx(\widehat \vy))]_+-\widehat \vy| \le \delta$}}
}
\Else{
Let $\va = \widehat\vy -\vtheta(\widehat\vx)$ {\color{blue}\algorithmiccomment{add a cutting plane $\langle\widehat\vy -\vtheta(\widehat\vx), \vy-\widehat\vy\rangle \le 0$}}\;
Set $\cE_{k+1}=\{\vy\in\RR^m: (\vy-\widehat\vy)^\top \vB^{-1} (\vy-\widehat\vy) \le 1\}$ with $\vB$ and $\widehat\vy$ updated by \eqref{eq:update-B-y}

}
}
Increase $k\gets k+1$.
}
}
\end{algorithm}

From Lemma~\ref{lem:cut-plane-multiple} and the property of the ellipsoid method (cf. \cite{bland1981ellipsoid}), we can show the finite convergence of Algorithm~\ref{alg:ellipsoid}, and furthermore, we can estimate its total complexity by Corollary~\ref{cor:iter-nesterov} if $\|\bar\vy\|\le b$.
\begin{theorem}\label{thm:ellipsoid}
Under Assumptions~\ref{assump:smooth}--\ref{assump:kkt}, Algorithm~\ref{alg:ellipsoid} will stop within at most $\left\lceil 2m(m+1)\log\frac{4b}{\eta}\right\rceil$ iterations, where $\eta$ is defined in Line~3 of the algorithm. If $\|\bar\vy\|\le b$, it must return $\mathrm{FLAG}=1$ and a vector $\widehat\vy\ge\vzero$ satisfying $\|[\vtheta(\vx(\widehat \vy))]_+-\widehat \vy\| \le \delta$ with at most $T$ evaluations of $f$, $\nabla f$, $\vtheta$, and $J_\vtheta$, where
\begin{equation}\label{eq:iter-ellipsoid}
T= \textstyle K\left(1+\lceil {\textstyle\log_{\gamma_1}\frac{L_\psi}{L_{\min}}}\rceil_+\right)\left(1+ 2\left\lceil \textstyle 2\sqrt{\frac{\gamma_1 L_\psi}{\mu}} \log\left(\frac{D_h}{\bar\vareps}\left(\sqrt{\gamma_1 L_\psi} + \frac{L_\psi}{\sqrt{L_{\min}}}\right)\sqrt{2\gamma_1 L_\psi+\mu}\right)\right\rceil_+\right),
\end{equation}
with $K=\left\lceil 2m(m+1)\log\frac{4b }{\eta}\right\rceil$, $L_\psi:=L_f+\beta b L_g $, and $\bar\vareps=\min\{\frac{\mu\delta}{4 B_g}, \frac{\mu^2\delta}{8(\mu B_g+\beta B_g^3)}\}$.
\end{theorem}

\begin{proof}
By the property of the ellipsoid method, we have (cf. \cite[Eq. 2.11]{bland1981ellipsoid}) 
$$\mathrm{vol}(\cE_{k}) \le e^{-\frac{1}{2(m+1)}}\mathrm{vol}(\cE_{k-1})\le e^{-\frac{k}{2(m+1)}}\mathrm{vol}(\cE_{0}),\,\forall\, k\ge1.$$ Hence, to satisfy the stopping condition $\mathrm{vol}(\cE_{k})\le 4^{-m}V_m(\eta)$, it suffices to have $e^{-\frac{k}{2(m+1)}}\mathrm{vol}(\cE_{0}) \le 4^{-m}V_m(\eta)$. Since $\cE_0$ is a ball of radius $b$, this requirement is equivalent to $e^{-\frac{k}{2(m+1)}}\le \left(\frac{\eta}{4b}\right)^m$, which holds if $k\ge \left\lceil 2m(m+1)\log\frac{4b }{\eta}\right\rceil$. We below estimate the number of evaluations of the function value and gradient.  

Notice that when Algorithm~\ref{alg:nesterov} is called, $\|\widehat\vy\|\le b$, and thus the smooth function $\psi$ has $(L_f+\beta L_g b)$-Lipschitz continuous gradient. Since Algorithm~\ref{alg:nesterov} is called at most $\left\lceil 2m(m+1)\log\frac{4b }{\eta}\right\rceil$ times, we have from Corollary~\ref{cor:iter-nesterov} that the total number of function and gradient evaluations is $T$ given in \eqref{eq:iter-ellipsoid}.  
\end{proof}

By Theorem~\ref{thm:ellipsoid}, we can guarantee to find a desired approximate solution $\widehat\vy$ by gradually increasing the search radius $b$. The algorithm is shown below.

\begin{algorithm}[h]
\caption{\textbf{S}earch by \textbf{t}he \textbf{E}llipsoid \textbf{M}ethod for $\max_{\vy\ge\vzero} d(\vy)$: $(\widehat\vx,\widehat\vy)=\mathrm{StEM}(\beta,\vz,\delta, L_{\min}, \gamma_1,\gamma_2)$}\label{alg:stem}
{\small
\DontPrintSemicolon
\textbf{Input:} multiplier vector $\vz\ge\vzero$, penalty $\beta>0$, target accuracy $\delta>0$, $L_{\min}>0$, and $\gamma_1>1, \gamma_2\ge 1$\;
\textbf{Overhead:} define $\vtheta(\vx) =\vg(\vx)+\frac{\vz}{\beta}$, $\Phi(\vx,\vy)$ as in \eqref{eq:equiv-sp}, and set $k=0$, $b_0=\frac{1}{\beta}$ and $\mathrm{FLAG}=0$.\; 
\While{$\mathrm{FLAG}=0$}{
Call Alg.~\ref{alg:ellipsoid}: $(\widehat\vx,\widehat\vy,\mathrm{FLAG})=\mathrm{Ellipsoid}(\beta,\vz,\delta, b_k, L_{\min}, \gamma_1,\gamma_2)$.\;
Let $b_{k+1}\gets 2b_k$ and increase $k\gets k+1$.
}
\textbf{Output} $(\widehat\vx,\widehat\vy)$.
}
\end{algorithm}

\begin{theorem}\label{thm:out-stem}
Under Assumptions~\ref{assump:smooth}--\ref{assump:kkt}, if $\delta\le \frac{8(\mu+\beta B_g^2)}{\beta\mu}$, then the output $(\widehat\vx,\widehat\vy)$ of Algorithm~\ref{alg:stem} must satisfy $\dist\big(\vzero,\partial_\vx\Phi(\widehat\vx, \widehat\vy)\big)\le\bar\vareps$, $\widehat\vy\ge\vzero$ and $\|[\vtheta(\vx(\widehat \vy))]_+-\widehat \vy\| \le \delta$, where $\bar\vareps=\min\{\frac{\mu\delta}{4 B_g}, \frac{\mu^2\delta}{8B_g(\mu +\beta B_g^2)}\}$. In addition, it needs at most $T$ evaluations of $f$, $\nabla f$, $\vtheta$, and $J_\vtheta$ to give the output, where
{\small\begin{equation}\label{eq:iter-stem}
T\le  \textstyle 3C K + 4C\sqrt{\gamma_1} \log\left(\frac{D_h}{\bar\vareps}\left(\sqrt{\gamma_1 L_{\max}} + \frac{L_{\max}}{\sqrt{L_{\min}}}\right)\sqrt{2\gamma_1 L_{\max}+\mu}\right)\left(K\sqrt\frac{L_f}{\mu} + \frac{\sqrt{L_g } \max\left\{1, \frac{2\sqrt{2\|\vz^*\|+\|\vz\|}} {\sqrt{2}-1}\right\}}{\sqrt\mu}\right),
\end{equation}
}
with the constants defined as 
\begin{align*}
&\textstyle L_{\max}=L_f+ L_g(4\|\vz^*\|+2\|\vz\|),\
C=2\left\lceil 2m(m+1)\log R\right\rceil \cdot\left(1+\lceil {\textstyle\log_{\gamma_1}\frac{L_{\max}}{L_{\min}}}\rceil_+\right),\\
&\textstyle K = \left\lceil \log_2(2\|\vz^*\|+\|\vz\|)\right\rceil_+ +1,\ R= \frac{64(2\|\vz^*\|+\|\vz\|)}{\beta}\left( \frac{4(\beta G +  4\|\vz^*\|+3\|\vz\|)(\mu+\beta B_g^2)^2}{\beta(\mu\delta)^2}+\frac{\mu+\beta B_g^2}{\mu\delta}\right).
\end{align*}
\end{theorem}

\begin{proof}
By the quadratic formula, we can easily have the positive root of \eqref{eq:cond-eta} to be
$$\textstyle \eta_+=\frac{\left(\frac{\mu\delta}{\mu+\beta B_g^2}\right)^2}{4\left(\sqrt{\frac{2B_d}{\beta}}+\sqrt{\frac{2B_d}{\beta}+\frac{\mu\delta}{\mu+\beta B_g^2}}\right)^2}\ge \frac{\left(\frac{\mu\delta}{\mu+\beta B_g^2}\right)^2}{8\left(\frac{4B_d}{\beta}+\frac{\mu\delta}{\mu+\beta B_g^2}\right)}.$$
Hence, it holds that
$$\textstyle \frac{b}{\eta_+}\le \frac{8b\left(\frac{4B_d}{\beta}+\frac{\mu\delta}{\mu+\beta B_g^2}\right)}{\left(\frac{\mu\delta}{\mu+\beta B_g^2}\right)^2}=8b\left( \frac{4B_d(\mu+\beta B_g^2)^2}{\beta(\mu\delta)^2}+\frac{\mu+\beta B_g^2}{\mu\delta}\right).$$
When $b\ge\frac{1}{\beta}$, the right hand side of the above inequality is greater than one by the assumption $\delta\le \frac{8(\mu+\beta B_g^2)}{\beta\mu}$, and since $\eta=\min\{\eta_+, b\}$ in Algorithm~\ref{alg:ellipsoid}, we have
\begin{equation}\label{eq:bd-ratio-b-eta}
\textstyle \frac{b}{\eta}=\max\{\frac{b}{\eta_+},1\}\le 8b\left( \frac{4B_d(\mu+\beta B_g^2)^2}{\beta(\mu\delta)^2}+\frac{\mu+\beta B_g^2}{\mu\delta}\right)\le 8b\left( \frac{4(\beta G + \|\vz\| + \beta b)(\mu+\beta B_g^2)^2}{\beta(\mu\delta)^2}+\frac{\mu+\beta B_g^2}{\mu\delta}\right),
\end{equation}
where we have used $\nabla d(\vy) = \beta(\vg(\vx(\vy)) + \frac{\vz}{\beta} -\vy)$ in \eqref{eq:grad-d} and thus the bound of $\nabla d(\vy)$ over $\cB_b^+$ satisfies $B_d\le \beta G + \|\vz\| + \beta b$ with $G$ defined in \eqref{eq:def-G-Bg}. 

Furthermore, by Lemma~\ref{lem:bound-feas-barx} and Theorem~\ref{thm:ellipsoid}, Algorithm~\ref{alg:ellipsoid} must return $\mathrm{FLAG}=1$ and a vector $\widehat\vy$ satisfying $\|[\vtheta(\vx(\widehat \vy))]_+-\widehat \vy\| \le \delta$ when $b\ge \frac{2\|\vz^*\|+\|\vz\|}{\beta}$. Since $b_0=\frac{1}{\beta}$ and $b_{k+1}=2b_k$, Algorithm~\ref{alg:stem} must stop after making at most $K$ calls to Algorithm~\ref{alg:ellipsoid}, where $K$ is the smallest positive integer such that $2^{K-1}\ge 2\|\vz^*\|+\|\vz\|$, i.e., $K=\left\lceil \log_2(2\|\vz^*\|+\|\vz\|)\right\rceil_+ +1$. In addition, from $b_{k+1}=2b_k$, it holds
\begin{equation}\label{eq:up-bd-bk}
\textstyle b_k = \frac{2^k}{\beta}< \frac{\max\{1,\, 4\|\vz^*\|+2\|\vz\|\}}{\beta},\text{ for each }0\le k\le K-1.
\end{equation} 

In the $k$-th call to Algorithm~\ref{alg:ellipsoid}, let $\eta_k$ denote the $\eta$ used  in Line~3 of Algorithm~\ref{alg:ellipsoid}, $L_{\psi_k}=L_f + \beta L_g b_k$ the gradient Lipschitz constant of the smooth function $\psi$, and $T_k$ the total number of gradient and function evaluations. Then, by \eqref{eq:up-bd-bk} and the definition of $L_{\max}$, we have $L_{\psi_k}\le L_{\max}$. Also, from \eqref{eq:bd-ratio-b-eta},  \eqref{eq:up-bd-bk}, and the definition of $R$, it follows $\frac{4 b_k}{\eta_k}\le R$ for each $0\le k\le K-1$. Moreover, we have from \eqref{eq:iter-ellipsoid} that
{\small\begin{align*}T_k\le
 &\, \textstyle \left\lceil 2m(m+1)\log R\right\rceil \left(1+\lceil {\textstyle\log_{\gamma_1}\frac{L_{\psi_k}}{L_{\min}}}\rceil_+\right)\left(1+ 2\left\lceil \textstyle 2\sqrt{\frac{\gamma_1 L_{\psi_k}}{\mu}} \log\left(\frac{D_h}{\bar\vareps}\left(\sqrt{\gamma_1 L_{\psi_k}} + \frac{L_{\psi_k}}{\sqrt{L_{\min}}}\right)\sqrt{2\gamma_1 L_{\psi_k}+\mu}\right)\right\rceil_+\right)\\
\le &\, \textstyle C \left(1+ 2\left\lceil \textstyle 2\sqrt{\frac{\gamma_1 L_{\psi_k}}{\mu}} \log\left(\frac{D_h}{\bar\vareps}\left(\sqrt{\gamma_1 L_{\max}} + \frac{L_{\max}}{\sqrt{L_{\min}}}\right)\sqrt{2\gamma_1 L_{\max}+\mu}\right)\right\rceil_+\right)\\
\le &\, 3C + 4C \textstyle\sqrt{\frac{\gamma_1 L_{\psi_k}}{\mu}} \log\left(\frac{D_h}{\bar\vareps}\left(\sqrt{\gamma_1 L_{\max}} + \frac{L_{\max}}{\sqrt{L_{\min}}}\right)\sqrt{2\gamma_1 L_{\max}+\mu}\right).
\end{align*}
}
Notice that $\sqrt{L_{\psi_k}}\le \sqrt{L_f} + \sqrt{\beta L_g b_k}$ and, thus 
\begin{align*}
\textstyle \sum_{k=0}^{K-1} \sqrt{L_{\psi_k}} \le K\sqrt{L_f} + \sum_{k=0}^{K-1}\sqrt{\beta L_g b_k} = &~ \textstyle K\sqrt{L_f}+\sqrt{L_g }\frac{\sqrt{2^K}-1}{\sqrt{2}-1}\\
\le &~ \textstyle K\sqrt{L_f}+ \sqrt{L_g } \max\left\{1, \frac{2\sqrt{2\|\vz^*\|+\|\vz\|}} {\sqrt{2}-1}\right\}.
\end{align*}
Therefore, $T$ must satisfy the condition in \eqref{eq:iter-stem} since $T\le \sum_{k=0}^{K-1} T_k$.
\end{proof}

\begin{remark}\label{rm:ellipsoid}
In terms of the dependence on $m$, the number $T$ in \eqref{eq:iter-stem} is proportional to $m^2$. We can improve it to the order of $m$ if a more advanced cutting-plane method is used, such as the volumetric-center cutting-plane method in \cite{vaidya1996new}, and the analytic-center cutting-plane method in \cite{atkinson1995cutting-analytic}, and the faster cutting plane method in \cite{lee2015faster}.
\end{remark}

\section{Overall iteration complexity of the first-order augmented Lagrangian method}\label{sec:overall}
In this section, we specify the implementation details in Algorithm~\ref{alg:ialm}. We use the method derived in section~\ref{sec:subroutine} as the subroutine to find each $\vx^{k+1}$. In addition, we choose a geometrically increasing sequence $\{\beta_k\}$ and stop the algorithm once an $\vareps$-KKT point is obtained. The pseudocode is given in Algorithm~\ref{alg:ialm-kkt}.

\begin{algorithm}[h]
\caption{Cutting-plane first-order iALM for problems in the form of \eqref{eq:ccp} with $m=O(1)$}\label{alg:ialm-kkt}
\DontPrintSemicolon
{\small
\textbf{Input:} $\beta_0>0$, $\sigma >1$, tolerance $\vareps>0$, $L_{\min}>0$, $\gamma_1>1$, and $\gamma_2\ge 1$\;
\textbf{Initialization:} choose $\vx^0\in\dom(h)$, and set $\vz^0=\vzero$ \;
\For{$k=0,1,\ldots $}{
Choose $\vareps_k\le \min\big\{\vareps,\, \frac{24 B_g(\mu + \beta_k B_g^2)}{\mu}\big\}$ and set $\delta_k = \frac{\vareps_k}{3\beta_k B_g}$.\;
\If{$m=1$}{
Call Alg.~\ref{alg:bisec}: $(\vx^{k+1},\vy^{k+1}) = \mathrm{BiSec}(\beta_k, \vz^k, \delta_k, L_{\min}, \gamma_1,\gamma_2)$
}
\Else{
Call Alg.~\ref{alg:stem}: $(\vx^{k+1},\vy^{k+1}) = \mathrm{StEM}(\beta_k, \vz^k, \delta_k, L_{\min}, \gamma_1,\gamma_2)$
}
\If{$m=1$ and $\frac{\mu}{4\beta_k B_g^2} > 1$, or $m>1$ and $\min\left\{\frac{\mu}{4\beta_k B_g^2},\, \frac{\mu^2}{8\beta_k B_g^2(\mu+\beta_k B_g^2)}\right\} > 1$}{
Call Alg.~\ref{alg:nesterov}: $\vx^{k+1}=\mathrm{APG}(\psi, h, \mu, L_{\min}, \vareps_k/3, \gamma_1,\gamma_2)$ with $\psi(\vx)=f(\vx) + \beta_k \big\langle \vy^{k+1}, \vg(\vx)\big\rangle$.
}
Update $\vz$ by $\vz^{k+1} =[\vz^k+\beta_k \vg(\vx^{k+1})]_+$.\;
Let $\beta_{k+1}\gets \sigma\beta_k$.\;
\If{$(\vx^{k+1},\vz^{k+1})$ is an $\vareps$-KKT point of \eqref{eq:ccp}}{
Output $(\bar\vx,\bar\vz)=(\vx^{k+1},\vz^{k+1})$ and stop
}
}
}
\end{algorithm}

The next theorem gives a bound on the number of calls to the subroutine.  
 
\begin{theorem}\label{thm:bound-num-subroutine}
Suppose that Assumptions~\ref{assump:smooth} through \ref{assump:kkt} hold. Let $(\beta_0, \sigma, \vareps, \gamma_1,\gamma_2)$ be the input of Algorithm~\ref{alg:ialm-kkt} and $\{(\vx^{k},\vy^k,\vz^k)\}_{k\ge0}$ be the generated sequence. Then $\dist\big(\vzero, \partial\cL_{\beta_k}(\vx^{k+1},\vz^k)\big) \le \vareps_k$ for each $k\ge0$. Suppose $\bar\vareps=\min\left\{\vareps,\, \sqrt{\frac{\vareps\mu(\sigma-1)}{8\sigma + 1}}\right\} \le \big\{\vareps,\, \frac{24 B_g(\mu + \beta_k B_g^2)}{\mu}\big\},\, \forall\, k\ge0$. Let $\vareps_k=\bar\vareps$ for all $k\ge 0$. Then after at most $K-1$ iterations, Algorithm~\ref{alg:ialm-kkt} will produce an $\vareps$-KKT point of \eqref{eq:ccp}, where
\begin{equation}\label{eq:max-K}
\textstyle K = \max\left\{\left\lceil \log_\sigma \frac{9\|\vz^*\|^2}{\beta_0\vareps}\right\rceil_+, \  \left\lceil\log_\sigma\frac{8\|\vz^*\|}{\beta_0\vareps}\right\rceil_+, \ \left\lceil\log_\sigma\frac{4}{\beta_0\vareps}\right\rceil_+\right\} + 1.
\end{equation}
In addition, the output multiplier vector $\bar\vz$ satisfies
\begin{equation}\label{eq:bd-out-z}
\textstyle \|\bar\vz\| \le 2\|\vz^*\|+\sqrt{\frac{2\sigma^2}{8\sigma + 1}}\max\big\{3\|\vz^*\|,\ 2\sqrt{2\|\vz^*\|},\ 2\big\}. 
\end{equation}
\end{theorem}

\begin{proof}
For each $k\ge0$, define 
$$\vtheta_k(\vx) = \vg(\vx) + \frac{\vz^k}{\beta_k},\ \phi_k(\vx) = F(\vx)+\frac{\beta_k}{2}\left\|{\textstyle[\vtheta_k(\vx)]_+}\right\|,\quad \Phi_k(\vx,\vy) = F(\vx)+\beta_k\left(\vy^\top\vtheta_k(\vx) - \frac{1}{2}\|\vy\|^2\right).$$
When $m=1$, if $(\vx^{k+1},\vy^{k+1})$ is obtained in Line~6 of Alg.~\ref{alg:ialm-kkt}, then we have from Theorem~\ref{thm:iter-bisec} that 
$$\textstyle \dist\big(\vzero,\partial_\vx\Phi_k(\vx^{k+1}, \vy^{k+1})\big)\le \frac{\mu\delta_k}{4B_g},\text{ and }~\big|[\vtheta_k(\vx(\vy^{k+1}))]_+ - \vy^{k+1}\big|\le \delta_k,$$
where $\vx(\vy^{k+1})=\argmin_\vx \Phi_k(\vx,\vy^{k+1})$. 
Furthermore, notice that if $\frac{\mu}{4\beta_k B_g^2}>1$, we will do Line~10 in Alg.~\ref{alg:ialm-kkt} to obtain a new $\vx^{k+1}$ that satisfies $\dist\big(\vzero,\partial_\vx\Phi_k(\vx^{k+1}, \vy^{k+1})\big)\le \frac{\vareps_k}{3}$. Now by Lemma~\ref{lem:approx-x-grad} and the choice of $\delta_k=\frac{\vareps_k}{3\beta_k B_g}$, we have $\dist\big(\vzero, \partial_\vx\cL_{\beta_k}(\vx^{k+1},\vz^k)\big) = \dist\big(\vzero, \partial \phi_k(\vx^{k+1})\big) \le \vareps_k.$

When $m>1$, by the choice of $\vareps_k$ and $\delta_k$, it holds $\delta_k\le \frac{8(\mu+\beta_k B_g^2)}{\beta_k \mu}$ for each $k$. Hence, we can use Theorem~\ref{thm:out-stem} and Lemma~\ref{lem:approx-x-grad} to show $\dist\big(\vzero, \partial_\vx\cL_{\beta_k}(\vx^{k+1},\vz^k)\big) \le \vareps_k$ by the same arguments as in the case of $m=1$. 

Therefore, for $m\ge1$, if $\vareps_k=\bar\vareps$ for all $k$, we have from Theorem~\ref{thm:rate-ialm} that the inequalities in \eqref{eq:feas-ialm} and \eqref{eq:cp-ialm} hold.  
By the choice of $\bar\vareps$, it holds $\frac{\bar\vareps^2(8\sigma+1)}{2\mu(\sigma-1)} \le \frac{\vareps}{2}$. Since $K-1\ge\log_\sigma \frac{9\|\vz^*\|^2}{\beta_0\vareps}$, then $\frac{9\|\vz^*\|^2}{2\beta_0\sigma^{K-1}} \le \frac{\vareps}{2}$, and thus we have from \eqref{eq:cp-ialm} that $\sum_{i=1}^m |z_i^K g_i(\vx^K)| \le \vareps$. In addition, noticing $\frac{\sqrt{2}(\sqrt\sigma+1))}{\sqrt{8\sigma + 1}}\le 1$ and $\bar\vareps\le\sqrt{\frac{\vareps\mu(\sigma-1)}{8\sigma + 1}}$, we have $\bar\vareps(\sqrt\sigma+1)\sqrt{\frac{2}{\mu(\sigma-1)}}\le \sqrt\vareps$, and thus \eqref{eq:feas-ialm} implies
$$\textstyle \big\|[\vg(\vx^{K})]_+\big\| \le \frac{4\|\vz^*\|}{\beta_0\sigma^{K-1}} + \frac{\sqrt\vareps }{\sqrt{\beta_0\sigma^{K-1}}}.$$
Now by the setting of $K$ in \eqref{eq:max-K}, we have that both terms on the right hand side of the above inequality are no greater than $\vareps/2$. Hence, $\|[\vg(\vx^K)]_+\| \le \vareps$, and thus $\vx^K$ must be an $\vareps$-KKT point of \eqref{eq:ccp}.

To show \eqref{eq:bd-out-z}, we have from the second inequality in \eqref{eq:bd-z-k-kkt} and the fact $\vareps_k=\bar\vareps\le \sqrt{\frac{\vareps\mu(\sigma-1)}{8\sigma + 1}},\forall\, k$ that
$$\textstyle \|\vz^k\|\le 2\|\vz^*\|+\sqrt{\frac{2\beta_0\bar\vareps^2}{\mu}\frac{\sigma^k-1}{\sigma-1}} \le 2\|\vz^*\|+\sqrt{\frac{2\beta_0\vareps\sigma^k}{8\sigma + 1}},\forall\, k\ge1.$$
Hence, for each $1\le k\le K$ with the $K$ given in \eqref{eq:max-K}, it holds
$$\textstyle \|\vz^k\|\le 2\|\vz^*\|+\sqrt{\frac{2\beta_0\vareps\sigma^K}{8\sigma + 1}}\le 2\|\vz^*\|+\sqrt{\frac{2\sigma^2}{8\sigma + 1}}\max\big\{3\|\vz^*\|,\ 2\sqrt{2\|\vz^*\|},\ 2\big\}.$$
Since the output $\bar\vz$ must be one of $\{\vz^k\}_{k=1}^K$, we complete the proof.
\end{proof}

By Theorem~\ref{thm:bound-num-subroutine}, we establish the overall iteration complexity of Algorithm~\ref{alg:ialm-kkt} to produce an $\vareps$-KKT point of \eqref{eq:ccp}. We first give the result for the case of $m=1$.

\begin{theorem}[Iteration complexity when $m=1$]\label{thm:iter-m=1}
Suppose that Assumptions~\ref{assump:smooth} through \ref{assump:kkt} hold, and $m=1$ in \eqref{eq:ccp}. Let $(\beta_0, \sigma, \vareps, \gamma_1,\gamma_2)$ be the input of Algorithm~\ref{alg:ialm-kkt} and $\{(\vx^{k},\vy^k,\vz^k)\}_{k\ge0}$ be the generated sequence. Suppose $\bar\vareps=\min\left\{\vareps,\, \sqrt{\frac{\vareps\mu(\sigma-1)}{8\sigma + 1}}\right\} \le \big\{\vareps,\, \frac{24 B_g(\mu + \beta_k B_g^2)}{\mu}\big\},\, \forall\, k\ge0$. Let $\vareps_k=\bar\vareps$ for all $k\ge 0$. Then Algorithm~\ref{alg:ialm-kkt} needs at most $T_{\mathrm{total}}=O\big(\sqrt{\frac{L_f+L_g(1+\|\vz^*\|)}{\mu}}|\log\vareps|^3\big)$ evaluations on $f$, $\nabla f$, $\vg$, and $J_\vg$ to produce an $\vareps$-KKT point of \eqref{eq:ccp}.
\end{theorem}

\begin{proof}
Let $K$ be the integer given in \eqref{eq:max-K} and $L_{\vz^k} = L_f+ L_g \max\{1, 4\|\vz^*\|+2 \|\vz^k\|\}$ for $0\le k \le K-1$. Also, let $T_k$ be the number of evaluations on $f$, $\nabla f$, $\vg$, and $J_\vg$ during the $k$-th iteration of Algorithm~\ref{alg:ialm-kkt}. From Theorem~\ref{thm:iter-bisec} and the setting $\delta_k=\frac{\vareps_k}{3\beta_k B_g}$, we have that the complexity incurred by Line~6 of Algorithm~\ref{alg:ialm-kkt} is $O(\sqrt{\frac{L_{\vz^k}}{\mu}}|\log\vareps|^2)$. In addition, the complexity incurred by Line~10 is $O\big(\sqrt{\frac{L_{\vz^k}}{\mu}}|\log\vareps|\big)$. From \eqref{eq:bd-z-k-kkt} with $\vareps_t=\bar\vareps,\forall\, t$, it follows $\|\vz^k\|=O(\|\vz^*\|)$, and thus $L_{\vz^k}= O(L_f + L_g(1+\|\vz^*\|))$ for $0\le k \le K-1$. Therefore, $T_k=O\big(\sqrt{\frac{L_f+L_g(1+\|\vz^*\|)}{\mu}}|\log\vareps|^2\big)$. Since $K=O(|\log\vareps|)$ in \eqref{eq:max-K}, the total complexity $T_{\mathrm{total}}=\sum_{k=0}^{K-1}T_k=O\big(\sqrt{\frac{L_f+L_g(1+\|\vz^*\|)}{\mu}}|\log\vareps|^3\big)$, which completes the proof.
\end{proof}

\begin{remark}\label{rm:m=1}
If $\beta_0$ is taken in the order of $\frac{1}{\vareps}$, then $K=O(1)$ in \eqref{eq:max-K}. In this case, the total complexity of Algorithm~\ref{alg:ialm-kkt} is $O\big(\sqrt{\frac{L_f+L_g(1+\|\vz^*\|)}{\mu}}|\log\vareps|^2\big)$ to produce an $\vareps$-KKT point.
\end{remark}

Similarly, we can show the complexity result for the case of $m>1$ by using Theorem~\ref{thm:out-stem}.
\begin{theorem}[Iteration complexity when $m>1$]\label{thm:iter-m>1}
Suppose that Assumptions~\ref{assump:smooth} through \ref{assump:kkt} hold, and $m>1$ in \eqref{eq:ccp}. Let $(\beta_0, \sigma, \vareps, \gamma_1,\gamma_2)$ be the input of Algorithm~\ref{alg:ialm-kkt} and $\{(\vx^{k},\vy^k,\vz^k)\}_{k\ge0}$ be the generated sequence. Suppose $\bar\vareps=\min\left\{\vareps,\, \sqrt{\frac{\vareps\mu(\sigma-1)}{8\sigma + 1}}\right\} \le \big\{\vareps,\, \frac{24 B_g(\mu + \beta_k B_g^2)}{\mu}\big\},\, \forall\, k\ge0$. Let $\vareps_k=\bar\vareps$ for all $k\ge 0$. Then Algorithm~\ref{alg:ialm-kkt} needs at most $T_{\mathrm{total}}=O\big(m^2\sqrt{\frac{L_f+L_g(1+\|\vz^*\|)}{\mu}}|\log\vareps|^3\big)$ evaluations on $f$, $\nabla f$, $\vg$, and $J_\vg$ to produce an $\vareps$-KKT point of \eqref{eq:ccp}.
\end{theorem}

\begin{remark}
Similar to Remark~\ref{rm:m=1}, the total complexity can be improved to $O\big(m^2\sqrt{\frac{L_f+L_g(1+\|\vz^*\|)}{\mu}}|\log\vareps|^2\big)$ if $\beta_0=\Theta(\frac{1}{\vareps})$. Ignoring the term $|\log\vareps|$, our result is better than the best known nonergodic complexity result $O\big(\sqrt{\frac{L_f+L_g(1+\|\vz^*\|)}{\mu\vareps}}|\log\vareps|\big)$ if $m=o(\vareps^{-\frac{1}{4}})$. As we discussed in Remark~\ref{rm:ellipsoid}, the dependence on $m^2$ can be improved to $m$ if a more advanced cutting plane method is used. In this case, we can obtain a result $O\big(m\sqrt{\frac{L_f+L_g(1+\|\vz^*\|)}{\mu}}|\log\vareps|^2\big)$ that is better than $O\big(\sqrt{\frac{L_f+L_g(1+\|\vz^*\|)}{\mu\vareps}}|\log\vareps|\big)$ if $m=o(\vareps^{-\frac{1}{2}})$ by ignoring the logarithmic term $|\log\vareps|$.
\end{remark}

\section{Extensions to convex or nonconvex problems}\label{sec:cvx-ncvx}
In this section, we extend the idea of the cutting-plane based FOM to constrained problems with a convex or nonconvex objective. Similar to the strongly convex case, we show that FOMs for solving problems with $O(1)$ nonlinear functional constraints can achieve a complexity result of almost the same order as for solving unconstrained problems.

\subsection{Extension to the convex case}\label{subsec:cvx-case}
We still consider the problem in \eqref{eq:ccp}. Suppose that the conditions in Assumptions~\ref{assump:smooth} and \ref{assump:cvx} hold. Instead of the strong convexity in Assumption~\ref{assump:scvx}, we assume the convexity of $f$ in this subsection.  

Given a target accuracy $\vareps>0$, to find an $\vareps$-KKT point of \eqref{eq:ccp}, we follow \cite{lan2016iteration-alm} and solve a perturbed strongly-convex problem:
\begin{equation}\label{eq:ccp-pert}
\min_{\vx\in\RR^n} F_\vareps(\vx):=f_\vareps(\vx) + h(\vx), \st \vg(\vx):=[g_1(\vx),\ldots, g_m(\vx)]\le\vzero,
\end{equation}
where
\begin{equation}\label{eq:f-eps}
f_\vareps(\vx) = f(\vx) + \frac{\vareps}{4D_h}\|\vx-\vx^0\|^2 \text{ with }\vx^0\in\dom(h).
\end{equation}
Let $\bar\vx\in \dom(h)$ be an $\frac{\vareps}{2}$-KKT point of \eqref{eq:ccp-pert}, i.e., there is $\bar\vz\ge\vzero$ such that
$$\textstyle \dist\left(\vzero, \partial_\vx\cL_0(\bar\vx, \bar\vz)+\frac{\vareps}{2D_h}(\bar\vx-\vx^0)\right) \le \frac{\vareps}{2}, \quad \|[\vg(\bar\vx)]_+\|\le \frac{\vareps}{2}, \quad \sum_{i=1}^m|\bar z_i g_i(\bar\vx)| \le \frac{\vareps}{2},$$
where $\cL_0$ is the Lagrange function of \eqref{eq:ccp}. Since $\|\frac{\vareps}{2D_h}(\bar\vx-\vx^0)\|\le \frac{\vareps}{2}$, $(\bar\vx,\bar\vz)$ must satisfy the conditions in \eqref{eq:eps-kkt}, and thus $\bar\vx$ is an $\vareps$-KKT point of \eqref{eq:ccp}. Based on this observation, we can apply Algorithm~\ref{alg:ialm-kkt} to the perturbed problem \eqref{eq:ccp-pert}. By Theorems~\ref{thm:iter-m=1} and \ref{thm:iter-m>1} and noticing that $f_\vareps$ in \eqref{eq:f-eps} is $\frac{\vareps}{2D_h}$-strongly convex, we obtain the following complexity result.

\begin{theorem}[complexity result for convex cases]
Assume that the conditions in Assumptions~\ref{assump:smooth} and \ref{assump:cvx} hold and that $f$ is convex. Given $\vareps>0$, suppose that the problem \eqref{eq:ccp-pert} has a KKT point $\vx_\vareps^*$ with a corresponding multiplier $\vz_\vareps^*$. Apply Algorithm~\ref{alg:ialm-kkt} to find an $\frac{\vareps}{2}$-KKT point $\bar\vx$ of  \eqref{eq:ccp-pert}. Then $\bar\vx$ is an $\vareps$-KKT point of \eqref{eq:ccp}, and the total number of evaluations on $f$, $\nabla f$, $\vg$, and $J_\vg$ is $O\big(m^2\sqrt{\frac{D_h\big(L_f+L_g(1+\|\vz_\vareps^*\|)\big)}{\vareps}}|\log\vareps|^3\big)$.
\end{theorem}

\subsection{Extension to the nonconvex case}
In this subsection, we assume Assumptions~\ref{assump:smooth} and \ref{assump:cvx} but do not assume the convexity of $f$. For the nonconvex case, we follow \cite{lin2019inexact-PP} and design an FOM within the framework of the proximal-point method, namely, we solve a sequence of problems in the form of
\begin{equation}\label{eq:ccp-pp-k}
\bar\vx^{k+1}\approx\argmin_{\vx\in\RR^n} \big\{F_k(\vx):=f(\vx) + L_f\|\vx-\bar\vx^k\|^2 + h(\vx), \st \vg(\vx):=[g_1(\vx),\ldots, g_m(\vx)]\le\vzero\big\},
\end{equation}
Under Assumptions~\ref{assump:smooth} and \ref{assump:cvx}, the above problem is convex, and its objective is $L_f$-strongly convex. Hence, we can apply Algorithm~\ref{alg:ialm-kkt} to find $\bar\vx^{k+1}$. Let $\vx_*^{k+1}$ be the unique optimal solution to \eqref{eq:ccp-pp-k}. To ensure the existence of a corresponding multiplier for each $k$ and also a uniform bound, we assume the Slater's condition on the original problem \eqref{eq:ccp}.

\begin{assumption}[Slater's condition]\label{assump:slater}
There is $\vx_{\mathrm{feas}}\in\mathrm{relint}(h)$ such that $g_i(\vx_{\mathrm{feas}})<0$ for all $i=1,\ldots,m$.
\end{assumption}

With the Slater's condition, the solution $\vx_*^{k+1}$ to \eqref{eq:ccp-pp-k} must be a KKT point (cf. \cite{rockafellar1970convex}). Let $\vz_*^{k+1}\ge\vzero$ be a corresponding multiplier. We give a uniform bound of $\vz_*^{k+1}$ below.

\begin{lemma}[uniform bound of multipliers]\label{lem:bd-mul}
Assume Assumptions~\ref{assump:smooth}, \ref{assump:cvx}, and \ref{assump:slater}. Let $\vx^*$ be a minimizer of \eqref{eq:ccp}, and let $\vx_*^{k+1}$ be the KKT point of \eqref{eq:ccp-pp-k} with a corresponding Lagrangian multiplier $\vz_*^{k+1}$. Then
\begin{equation}\label{eq:bd-mul}
\textstyle \|\vz_*^{k+1}\|\le B_\vz:=\frac{F(\vx_{\mathrm{feas}}) - F(\vx^*) + L_fD_h^2}{\min_i \big(-g_i(\vx_{\mathrm{feas}})\big)},\forall\, k\ge0.
\end{equation}
\end{lemma}

\begin{proof}
From the KKT system, we have that
\begin{equation}\label{eq:kkt-ppm-k}
-\sum_{i=1}^m (z_*^{k+1})_i\nabla g_i(\vx_*^{k+1}) \in \partial F_k(\vx_*^{k+1}),\quad (z_*^{k+1})_ig_i(\vx_*^{k+1}) = 0, \forall\, i=1,\ldots,m.
\end{equation}
Then we have
\begin{align}\label{eq:uf-bd-t1}
\textstyle \sum_{i=1}^m(z_*^{k+1})_i g_i(\vx_{\mathrm{feas}})\ge &~ \textstyle \sum_{i=1}^m(z_*^{k+1})_i\Big(g_i(\vx_*^{k+1})+\big\langle \vx_{\mathrm{feas}}-\vx_*^{k+1}, \nabla g_i(\vx_*^{k+1})\big\rangle\Big)\cr
=&~ \textstyle \left\langle \vx_{\mathrm{feas}}-\vx_*^{k+1}, \sum_{i=1}^m(z_*^{k+1})_i\nabla g_i(\vx_*^{k+1})\right\rangle\cr
\ge &~ F_k(\vx_*^{k+1})- F_k(\vx_{\mathrm{feas}}),
\end{align}
where the first inequality is from the convexity of each $g_i$ and the nonnegativity of $\vz_*^{k+1}$, the equality holds because of the second equation in \eqref{eq:kkt-ppm-k}, and the last inequality follows from the convexity of $F_k$ and the first equation in \eqref{eq:kkt-ppm-k}. 

Since the diameter of $\dom(h)$ is $D_h$, it holds that
\begin{align}\label{eq:uf-bd-t2}
-F_k(\vx_*^{k+1})+ F_k(\vx_{\mathrm{feas}})  = &~F(\vx_{\mathrm{feas}}) + L_f\|\vx_{\mathrm{feas}}-\bar\vx^k\|^2 - F(\vx_*^{k+1})-L_f\|\vx_*^{k+1}-\bar\vx^k\|^2\cr
\le &~ F(\vx_{\mathrm{feas}}) - F(\vx_*^{k+1}) + L_fD_h^2.
\end{align}
Notice $F(\vx_*^{k+1})\ge F(\vx^*)$. Hence, $F(\vx_{\mathrm{feas}}) - F(\vx_*^{k+1})\le F(\vx_{\mathrm{feas}}) - F(\vx^*)$, and from \eqref{eq:uf-bd-t2}, it follows that $-F_k(\vx_*^{k+1})+ F_k(\vx_{\mathrm{feas}})\le F(\vx_{\mathrm{feas}}) - F(\vx^*) + L_fD_h^2$. Now we have from \eqref{eq:uf-bd-t1} that
$$\textstyle \|\vz_*^{k+1}\|_1 \le \frac{-F_k(\vx_*^{k+1})+ F_k(\vx_{\mathrm{feas}})}{\min_i \big(-g_i(\vx_{\mathrm{feas}})\big)}\le \frac{F(\vx_{\mathrm{feas}}) - F(\vx^*) + L_fD_h^2}{\min_i \big(-g_i(\vx_{\mathrm{feas}})\big)},$$
and we complete the proof by $\|\vz_*^{k+1}\|_2\le \|\vz_*^{k+1}\|_1$.
\end{proof}

Similar to our discussion in section~\ref{subsec:cvx-case}, we notice that if $\bar\vx^{k+1}$ is an $\frac{\vareps}{2}$-KKT point of \eqref{eq:ccp-pp-k} and also $2L_f\|\bar\vx^{k+1}-\bar\vx^k\|\le \frac{\vareps}{2}$, then $\bar\vx^{k+1}$ is an $\vareps$-KKT point of \eqref{eq:ccp}. Below, we show that the sum of $\|\bar\vx^{k+1}-\bar\vx^k\|^2$ can be controlled if each $\bar\vx^{k+1}$ is obtained with sufficient accuracy, and thus a near-KKT point of \eqref{eq:ccp} can be produced.

\begin{theorem}[complexity result for nonconvex cases]\label{thm:ncvx-case}
Assume Assumptions~\ref{assump:smooth}, \ref{assump:cvx}, and \ref{assump:slater}. Let $\vx^*$ be a minimizer of \eqref{eq:ccp}. Let $\vareps>0$ be given and $\bar\vx^0\in\dom(h)$. Generate the sequence $\{(\bar\vx^k,\bar\vz^k)\}_{k\ge1}$ by applying Algorithm~\ref{alg:ialm-kkt} to \eqref{eq:ccp-pp-k} with the target accuracy $\tilde\vareps=\min\big\{\frac{\vareps}{2},\ \frac{\vareps^2}{64L_f(D_h+2\bar B_\vz)}\big\}$, where
\begin{equation}\label{eq:barBz}
\textstyle \bar B_\vz := 2B_\vz+\sqrt{\frac{2\sigma^2}{8\sigma + 1}}\max\big\{3B_\vz,\ 2\sqrt{2B_\vz},\ 2\big\},
\end{equation} 
with $B_\vz$ defined in \eqref{eq:bd-mul}. 
Then after solving at most $K$ proximal point subproblems as that in \eqref{eq:ccp-pp-k}, we can find an $\vareps$-KKT point of \eqref{eq:ccp}, where  
\begin{equation}\label{eq:ppm-K}
\textstyle K =\left\lceil \frac{64L_f(F(\bar\vx^0)-F(\vx^*)+L_fD_h^2+\bar B_\vz\|[\vg(\bar\vx^0)]_+\|)}{\vareps^2}\right\rceil.
\end{equation}
In addition, the total number of evaluations on on $f$, $\nabla f$, $\vg$, and $J_\vg$ is $O\big(\frac{m^2}{\vareps^2}|\log\vareps|^3\big)$.
\end{theorem}

\begin{proof}
Since each $(\bar\vx^{k+1},\bar\vz^{k+1})$ is an output from Algorithm~\ref{alg:ialm-kkt} applied to \eqref{eq:ccp-pp-k} and with a target accuracy $\tilde\vareps$, then $\bar\vx^{k+1}$ is an $\tilde\vareps$-KKT point of the problem in \eqref{eq:ccp-pp-k}, and thus there is a subgradient $\tilde\nabla F_k(\vx^{k+1})\in \partial F_k(\bar\vx^{k+1})$ such that
\begin{equation}\label{eq:eps-k+1}
\|\tilde\nabla F_k(\bar\vx^{k+1}) + J_\vg^\top(\bar\vx^{k+1})\bar\vz^{k+1}\| \le \tilde\vareps,\quad \|\vg(\bar\vx^{k+1})\|\le \tilde\vareps, \forall\,k\ge0.
\end{equation}
From the first inequality in \eqref{eq:eps-k+1} and recalling that the diameter of $\dom(h)$ is $D_h$, we have
$$\left\langle \bar\vx^{k+1}-\bar\vx^k, \tilde\nabla F_k(\bar\vx^{k+1}) + J_\vg^\top(\bar\vx^{k+1})\bar\vz^{k+1}\right\rangle\le D_h\tilde\vareps.$$
Hence, by the $L_f$-strong convexity of $F_k$ and convexity of each $g_i$, we have
\begin{align}\label{eq:bd-diff-fk}
D_h\tilde\vareps\ge&~\left\langle \bar\vx^{k+1}-\bar\vx^k, \tilde\nabla F_k(\bar\vx^{k+1}) + J_\vg^\top(\bar\vx^{k+1})\bar\vz^{k+1}\right\rangle\cr
\ge &~ F_k(\bar\vx^{k+1})-F_k(\bar\vx^k) + \frac{L_f}{2}\|\bar\vx^{k+1}-\bar\vx^k\|^2 + \langle \bar\vz^{k+1}, \vg(\bar\vx^{k+1}) - \vg(\bar\vx^k)\rangle\cr
=&~ F(\bar\vx^{k+1})-F(\bar\vx^k) + L_f\|\bar\vx^{k+1}-\bar\vx^k\|^2+ \langle \bar\vz^{k+1}, \vg(\bar\vx^{k+1}) - \vg(\bar\vx^k)\rangle.
\end{align}

By \eqref{eq:bd-out-z} and \eqref{eq:bd-mul}, we have $\|\bar\vz^{k+1}\| \le \bar B_\vz,\forall\, k\ge0$, where $\bar B_\vz$ is given in \eqref{eq:barBz}. Hence, it follows from the second inequality in \eqref{eq:eps-k+1} that 
$\langle \bar\vz^{k+1}, \vg(\bar\vx^{k+1}) - \vg(\bar\vx^k)\rangle\ge - 2\tilde\vareps\bar B_\vz, \forall\, k\ge1$.  
Now summing up \eqref{eq:bd-diff-fk}, we obtain
\begin{equation}\label{eq:bd-sum-xk-xk1}
\textstyle L_f\sum_{k=0}^{K-1}\|\bar\vx^{k+1}-\bar\vx^k\|^2 \le KD_h\tilde\vareps + F(\bar\vx^0)-F(\bar\vx^K)+(2K-1)\tilde\vareps\bar B_\vz+\bar B_\vz\|[\vg(\bar\vx^0)]_+\|,
\end{equation}
where we have used $\langle\bar\vz^1, \vg(\bar\vx^0)\rangle\le \|\bar\vz^1\|\cdot\|[\vg(\bar\vx^0)]_+\|\le \bar B_\vz\|[\vg(\bar\vx^0)]_+\|$.

Because $\vx_*^K$ is a KKT-point of \eqref{eq:ccp-pp-k} with a corresponding multiplier $\vz_*^K$, we have from \eqref{eq:opt-ineq} that
$$F_{K-1}(\bar\vx^K) - F_{K-1}(\vx_*^K) + \big\langle\vz_*^K, \vg(\bar\vx^K)\big\rangle \ge0.$$
Plugging $F_{K-1}(\cdot)=F(\cdot)+L_f\|\cdot-\bar\vx^{K-1}\|^2$ into the above equation gives
$$F(\bar\vx^K) + L_f\|\bar\vx^K-\bar\vx^{K-1}\|^2 - F(\vx_*^K) - L_f\|\vx_*^K-\bar\vx^{K-1}\|^2 + \big\langle\vz_*^K, \vg(\bar\vx^K)\big\rangle \ge0.$$
Now using \eqref{eq:bd-mul}, $\|\vg(\bar\vx^K)\|\le\tilde\vareps$, $\|\bar\vx^K-\bar\vx^{K-1}\|^2\le D_h^2$, and the fact $F(\vx_*^K)\ge F(\vx^*)$,  we have from the above inequality that
$-F(\bar\vx^K)\le -F(\vx^*) + L_fD_h^2 + \tilde\vareps B_\vz.$
This inequality together with \eqref{eq:bd-sum-xk-xk1} gives
\begin{equation}\label{eq:bd-sum-xk-xk1-2}
\textstyle L_f\sum_{k=0}^{K-1}\|\bar\vx^{k+1}-\bar\vx^k\|^2 \le KD_h\tilde\vareps + F(\bar\vx^0)-F(\vx^*)+L_fD_h^2+2K\tilde\vareps\bar B_\vz+\bar B_\vz\|[\vg(\bar\vx^0)]_+\|.
\end{equation}
Multiplying $L_f$ to both sides of the above inequality and taking square root, we have
\begin{equation}\label{eq:min-bd-xk-xk1}
\min_{0\le k<K}L_f\|\bar\vx^{k+1}-\bar\vx^k\|\le \textstyle\sqrt{L_f(D_h\tilde\vareps+2\bar B_\vz\tilde\vareps)} + \sqrt{\frac{L_f\big(F(\bar\vx^0)-F(\vx^*)+L_fD_h^2+\tilde\vareps\|[\vg(\bar\vx^0)]_+\|\big)}{K}}.
\end{equation}
Therefore, by the setting of $\tilde\vareps$ and $K$, we have $\min_{0\le k<K}L_f\|\bar\vx^{k+1}-\bar\vx^k\|\le \frac{\vareps}{4}$. Suppose $L_f\|\bar\vx^{k_0+1}-\bar\vx^{k_0}\|\le \frac{\vareps}{4}$. Then by our discussion above Theorem~\ref{thm:ncvx-case}, $\bar\vx^{k_0+1}$ is an $\vareps$-KKT point of \eqref{eq:ccp}. From Theorems~\ref{thm:iter-m=1} and \ref{thm:iter-m>1}, the complexity of solving one problem as that in \eqref{eq:ccp-pp-k} is $O(m^2|\log\vareps|^3)$, and thus the total complexity is $O(Km^2|\log\vareps|^3)=O(\frac{m^2}{\vareps^2}|\log\vareps|^3)$. This completes the proof.
\end{proof}

\section{Experimental results}\label{sec:numerical}
In this section, we demonstrate the established theory by performing numerical experiments on solving quadratically-constrained quadratic program (QCQP):
\begin{equation}\label{eq:qcqp}
\min_{\vx\in\RR^n} \textstyle \frac{1}{2}\vx^\top\vQ_0\vx + \vx^\top\vc_0, \st \frac{1}{2}\vx^\top\vQ_j\vx + \vx^\top\vc_j + d_j \le 0, \, j=1,\ldots,m;\ x_i\in[l_i,u_i],\, i=1,\ldots,n.
\end{equation}
In the experiment, $\vQ_0$ is generated to be positive definite, $\vQ_j$ is positive semidefinite but rank-deficient for each $j=1,\ldots,m$, and $l_i=-10$ and $u_i=10$ for each $i$. All $d_j$ are negative so the Slater's condition holds.

We compare two implementations of the iALM in Algorithm~\ref{alg:ialm}. One directly applies the APG method in Algorithm~\ref{alg:nesterov} to solve each ALM subproblem, and we call it ``APG-based iALM''. The other uses the proposed cutting-plane based FOM to solve subproblems, namely, we implement Algorithm~\ref{alg:ialm-kkt} to solve \eqref{eq:qcqp}, and we call it ``cutting-plane iALM''. For both implementations, we set $\beta_k=10^{k-1}$ for each outer iteration $k\ge1$ and run the iALM to 5 outer iterations. The target accuracy for a near-KKT point is set to $\vareps=10^{-4}$. In the implementation of the APG-based iALM, due to the quadratic penalty term, we apply Algorithm~\ref{alg:nesterov} with line search for a local smoothness constant and set the parameters to $\gamma_1=1.5, \gamma_2=2, L_{\min}=1$. In the implementation of the cutting-plane iALM, we use Algorithm~\ref{alg:nesterov} to solve problems in the form of \eqref{eq:def-x-y}, for which we can explicitly compute the global smoothness constant, and thus we simply set $L_{\min}$ to the global smoothness constant.

We test three groups of QCQP instances, each of which has $n=1000$. The first group has $m=1$ constraint, the second has $m=2$, and the third has $m=5$. For each group, we conduct 5 independent trials. For each instance, we report the number of gradient and function evaluations, the primal residual, dual residual, and complementarity violation, which are denoted as \#grad, \#func, pres, dres, and compl,  for solving each ALM subproblem. In order to demonstrate the worst-case theoretical result, we use randomly-generated initial point while solving each ALM subproblem. The performance of the iALM can be much better if the warm-start technique is adopted. The results are shown in Tables~\ref{table:m1}--\ref{table:m5}. For the cutting-plane iALM, its \#func. is zero and not shown in the tables, because we feed the APG an explicitly-computed smoothness constant and no line search is performed. 

From the results, we see that as the penalty parameter increases, the APG-based iALM needs significantly more iterations to solve the subproblems, while the cutting-plane iALM does not suffer from the big penalty parameter. However, the cutting-plane iALM has worse scalability to $m$, and this matches with our theory.

\begin{table}\caption{Results by the APG based first-order iALM and the proposed cutting-plane based first-order iALM for solving QCQP \eqref{eq:qcqp} with $m=1$ and $n=1000$.}\label{table:m1}
\begin{center}
\resizebox{0.9\textwidth}{!}{\begin{tabular}{|c|c||ccccc|cccc|}
\hline
\multicolumn{2}{|c||}{} & \multicolumn{5}{|c|}{APG-based iALM} & \multicolumn{4}{|c|}{proposed cutting-plane iALM}\\\hline\hline
out.Iter & $\beta$ & \#grad & \#func & pres & dres & compl & \#grad  & pres & dres & compl\\\hline\hline
\multicolumn{2}{|c}{trial 1} & \multicolumn{5}{c}{ total running time = 1592 sec.} & \multicolumn{4}{c|}{total running time = 25 sec.}\\\hline
1 & 1 & 6281 & 9188 & 5.71e-02 & 9.70e-05 & 3.26e-03 & 3370 & 2.17e-02 & 1.18e-10 & 4.71e-04\\
2 & 10 & 21047 & 30762 & 1.08e-06 & 9.87e-05 & 6.16e-08 & 2168 & 1.83e-07 & 1.34e-09 & 3.98e-09\\
3 & $10^2$ & 69584 & 101670 & 0.00e+00 & 9.87e-05 & 1.60e-09 & 1626 & 0.00e+00 & 1.04e-09 & 8.96e-11\\
4 & $10^3$ & 226083 & 330294 & 9.80e-10 & 9.97e-05 & 5.60e-11 & 1630 & 1.02e-09 & 1.88e-08 & 2.20e-11\\
5 & $10^4$ & 735386 & 1074310 & 0.00e+00 & 1.00e-04 & 3.01e-11 & 1638 & 0.00e+00 & 7.82e-11 & 1.90e-12\\\hline\hline
\multicolumn{2}{|c}{trial 2} & \multicolumn{5}{c}{ total running time = 1609 sec.} & \multicolumn{4}{c}{total running time = 24 sec.}\\\hline
1 & 1 & 6307 & 9226 & 4.37e-02 & 9.56e-05 & 1.91e-03 & 3288 & 3.99e-02 & 1.87e-09 & 1.59e-03\\
2 & 10 & 21362 & 31222 & 1.74e-06 & 9.97e-05 & 7.60e-08 & 2116 & 3.61e-07 & 1.37e-10 & 1.44e-08\\
3 & $10^2$ & 70137 & 102478 & 0.00e+00 & 9.98e-05 & 6.09e-09 & 1572 & 0.00e+00 & 7.15e-09 & 1.83e-10\\
4 & $10^3$ & 228676 & 334082 & 1.96e-09 & 9.92e-05 & 8.59e-11 & 1576 & 0.00e+00 & 2.43e-08 & 1.79e-11\\
5 & $10^4$ & 740405 & 1081642 & 3.44e-10 & 9.95e-05 & 1.50e-11 & 1586 & 1.02e-10 & 4.36e-11 & 4.08e-12\\\hline\hline
\multicolumn{2}{|c}{trial 3} & \multicolumn{5}{c}{ total running time = 1698 sec.} & \multicolumn{4}{c|}{total running time = 23 sec.}\\\hline
1 & 1 & 6704 & 9806 & 4.78e-02 & 9.95e-05 & 2.29e-03 & 3248 & 4.58e-02 & 4.30e-09 & 2.10e-03\\
2 & 10 & 22390 & 32724 & 0.00e+00 & 9.93e-05 & 6.48e-09 & 1980 & 5.01e-07 & 2.90e-09 & 2.29e-08\\
3 & $10^2$ & 72693 & 106212 & 8.53e-09 & 9.94e-05 & 4.08e-10 & 1470 & 2.50e-09 & 8.67e-09 & 1.14e-10\\
4 & $10^3$ & 240491 & 351342 & 1.41e-09 & 9.96e-05 & 6.75e-11 & 1480 & 8.43e-11 & 1.13e-08 & 3.86e-12\\
5 & $10^4$ & 778628 & 1137480 & 0.00e+00 & 9.97e-05 & 1.81e-11 & 1478 & 0.00e+00 & 4.00e-11 & 3.29e-12\\\hline\hline
\multicolumn{2}{|c}{trial 4} & \multicolumn{5}{c}{ total running time = 1679 sec.} & \multicolumn{4}{c|}{total running time = 23 sec.}\\\hline
1 & 1 & 6619 & 9682 & 4.31e-02 & 9.50e-05 & 1.86e-03 & 3204 & 3.88e-02 & 8.20e-10 & 1.51e-03\\
2 & 10 & 22134 & 32350 & 5.86e-07 & 9.71e-05 & 2.52e-08 & 2054 & 3.90e-07 & 9.17e-09 & 1.51e-08\\
3 & $10^2$ & 72834 & 106418 & 0.00e+00 & 9.83e-05 & 1.15e-09 & 1578 & 0.00e+00 & 1.47e-09 & 1.38e-10\\
4 & $10^3$ & 239419 & 349776 & 1.54e-09 & 9.99e-05 & 6.64e-11 & 1582 & 6.89e-10 & 2.23e-08 & 2.68e-11\\
5 & $10^4$ & 776840 & 1134868 & 0.00e+00 & 9.99e-05 & 1.98e-11 & 1582 & 0.00e+00 & 4.40e-09 & 5.20e-12\\\hline\hline
\multicolumn{2}{|c}{trial 5} & \multicolumn{5}{c}{ total running time = 1650 sec.} & \multicolumn{4}{c|}{total running time = 26 sec.}\\\hline
1 & 1 & 6541 & 9568 & 5.14e-02 & 9.63e-05 & 2.65e-03 & 3134 & 5.06e-02 & 1.34e-09 & 2.56e-03\\
2 & 10 & 22104 & 32306 & 0.00e+00 & 9.87e-05 & 3.14e-08 & 2138 & 4.74e-07 & 2.09e-10 & 2.40e-08\\
3 & $10^2$ & 71910 & 105068 & 6.45e-08 & 9.96e-05 & 3.32e-09 & 1538 & 0.00e+00 & 3.72e-13 & 4.81e-11\\
4 & $10^3$ & 235216 & 343636 & 6.21e-09 & 9.94e-05 & 3.19e-10 & 1542 & 1.19e-09 & 3.55e-08 & 6.03e-11\\
5 & $10^4$ & 766509 & 1119776 & 0.00e+00 & 9.99e-05 & 6.33e-12 & 1548 & 0.00e+00 & 4.53e-09 & 3.87e-12\\\hline
\end{tabular}
}
\end{center}
\end{table}

\begin{table}\caption{Results by the APG based first-order iALM and the proposed cutting-plane based first-order iALM for solving QCQP \eqref{eq:qcqp} with $m=2$ and $n=1000$.}\label{table:m2}
\begin{center}
\resizebox{0.9\textwidth}{!}{\begin{tabular}{|c|c||ccccc|cccc|}
\hline
\multicolumn{2}{|c||}{} & \multicolumn{5}{|c|}{APG based iALM} & \multicolumn{4}{|c|}{proposed cutting-plane iALM}\\\hline\hline
out.Iter & $\beta$ & \#grad & \#func & pres & dres & compl & \#grad  & pres & dres & compl\\\hline\hline
\multicolumn{2}{|c}{trial 1} & \multicolumn{5}{c}{ total running time = 2651 sec.} & \multicolumn{4}{c|}{total running time = 114 sec.}\\\hline
1 & 1 & 6615 & 9676 & 5.40e-02 & 9.89e-05 & 2.09e-03 & 6534 & 5.09e-02 & 8.97e-09 & 1.94e-03\\
2 & 10 & 21978 & 32122 & 5.96e-07 & 9.97e-05 & 2.15e-08 & 6560 & 4.89e-07 & 6.98e-09 & 1.83e-08\\
3 & $10^2$ & 72140 & 105404 & 3.21e-09 & 1.00e-04 & 5.68e-10 & 6590 & 3.57e-09 & 3.80e-09 & 1.50e-10\\
4 & $10^3$ & 235420 & 343934 & 1.95e-09 & 9.89e-05 & 8.61e-11 & 6634 & 0.00e+00 & 7.07e-10 & 1.63e-11\\
5 & $10^4$ & 766572 & 1119868 & 1.19e-10 & 9.96e-05 & 8.75e-12 & 6652 & 5.51e-11 & 3.77e-09 & 1.73e-12\\\hline\hline
\multicolumn{2}{|c}{trial 2} & \multicolumn{5}{c}{ total running time = 2648 sec.} & \multicolumn{4}{c|}{total running time = 113 sec.}\\\hline
1 & 1 & 6652 & 9730 & 5.76e-02 & 9.72e-05 & 2.46e-03 & 6594 & 5.47e-02 & 7.64e-09 & 2.16e-03\\
2 & 10 & 22145 & 32366 & 6.59e-07 & 9.99e-05 & 3.03e-08 & 6722 & 5.16e-07 & 9.72e-09 & 2.07e-08\\
3 & $10^2$ & 72255 & 105572 & 1.01e-08 & 9.97e-05 & 4.97e-10 & 6736 & 0.00e+00 & 4.14e-09 & 2.06e-10\\
4 & $10^3$ & 234871 & 343132 & 4.62e-09 & 9.97e-05 & 2.20e-10 & 6782 & 1.18e-10 & 1.98e-09 & 2.24e-11\\
5 & $10^4$ & 763890 & 1115950 & 2.56e-10 & 9.99e-05 & 1.41e-11 & 6822 & 8.47e-11 & 9.78e-09 & 3.01e-12\\\hline\hline
\multicolumn{2}{|c}{trial 3} & \multicolumn{5}{c}{ total running time = 2774 sec.} & \multicolumn{4}{c|}{total running time = 109 sec.}\\\hline
1 & 1 & 6986 & 10218 & 6.98e-02 & 9.94e-05 & 3.53e-03 & 6418 & 5.91e-02 & 9.39e-09 & 2.74e-03\\
2 & 10 & 23158 & 33846 & 1.11e-06 & 9.82e-05 & 4.94e-08 & 6472 & 5.68e-07 & 5.57e-09 & 2.53e-08\\
3 & $10^2$ & 75312 & 110038 & 9.48e-09 & 9.87e-05 & 4.33e-10 & 6506 & 1.01e-10 & 5.59e-09 & 4.54e-10\\
4 & $10^3$ & 245766 & 359048 & 1.90e-09 & 9.95e-05 & 1.70e-10 & 6524 & 5.87e-10 & 6.93e-09 & 3.54e-11\\
5 & $10^4$ & 796022 & 1162890 & 1.37e-10 & 9.98e-05 & 6.25e-12 & 6566 & 2.75e-11 & 3.44e-09 & 9.53e-13\\\hline\hline
\multicolumn{2}{|c}{trial 4} & \multicolumn{5}{c}{ total running time = 2817 sec.} & \multicolumn{4}{c|}{total running time = 110 sec.}\\\hline
1 & 1 & 7038 & 10294 & 6.39e-02 & 9.95e-05 & 3.00e-03 & 6100 & 5.56e-02 & 3.99e-10 & 2.28e-03\\
2 & 10 & 23247 & 33976 & 1.09e-06 & 9.85e-05 & 5.16e-08 & 6168 & 5.25e-07 & 3.74e-11 & 2.21e-08\\
3 & $10^2$ & 76117 & 111214 & 0.00e+00 & 9.86e-05 & 7.61e-10 & 6182 & 5.54e-09 & 8.63e-09 & 2.23e-10\\
4 & $10^3$ & 248129 & 362500 & 2.93e-09 & 9.91e-05 & 1.27e-10 & 6194 & 5.10e-11 & 7.86e-09 & 2.09e-11\\
5 & $10^4$ & 806468 & 1178150 & 0.00e+00 & 9.99e-05 & 1.25e-11 & 6254 & 3.24e-11 & 1.39e-10 & 2.63e-12\\\hline\hline
\multicolumn{2}{|c}{trial 5} & \multicolumn{5}{c}{ total running time = 2690 sec.} & \multicolumn{4}{c|}{total running time = 110 sec.}\\\hline
1 & 1 & 6715 & 9822 & 6.15e-02 & 9.42e-05 & 2.74e-03 & 6210 & 5.91e-02 & 2.97e-09 & 2.67e-03\\
2 & 10 & 22286 & 32572 & 7.74e-07 & 9.96e-05 & 3.51e-08 & 6206 & 4.85e-07 & 4.42e-09 & 2.17e-08\\
3 & $10^2$ & 73264 & 107046 & 5.20e-08 & 9.94e-05 & 3.11e-09 & 6242 & 8.49e-09 & 5.36e-09 & 3.79e-10\\
4 & $10^3$ & 237846 & 347478 & 6.31e-09 & 9.95e-05 & 4.05e-10 & 6242 & 2.13e-10 & 8.53e-09 & 2.85e-11\\
5 & $10^4$ & 772485 & 1128506 & 1.27e-10 & 9.99e-05 & 5.95e-12 & 6328 & 6.50e-12 & 1.66e-09 & 2.25e-12\\\hline
\end{tabular}
}
\end{center}
\end{table}

\begin{table}\caption{Results by the APG based first-order iALM and the proposed cutting-plane based first-order iALM for solving QCQP \eqref{eq:qcqp} with $m=5$ and $n=1000$.}\label{table:m5}
\begin{center}
\resizebox{0.9\textwidth}{!}{\begin{tabular}{|c|c||ccccc|cccc|}
\hline
\multicolumn{2}{|c||}{} & \multicolumn{5}{|c|}{APG based iALM} & \multicolumn{4}{|c|}{proposed cutting-plane iALM}\\\hline\hline
out.Iter & $\beta$ & \#grad & \#func & pres & dres & compl & \#grad  & pres & dres & compl\\\hline\hline
\multicolumn{2}{|c}{trial 1} & \multicolumn{5}{c}{ total running time = 6190 sec.} & \multicolumn{4}{c|}{total running time = 740 sec.}\\\hline
t1 & 1 & 7075 & 10348 & 7.77e-02 & 9.93e-05 & 2.97e-03 & 32538 & 7.77e-02 & 4.21e-09 & 2.97e-03\\
2 & 10 & 23340 & 34112 & 1.72e-06 & 9.92e-05 & 7.07e-08 & 32744 & 8.60e-07 & 6.13e-09 & 3.33e-08\\
3 & $10^2$ & 76514 & 111794 & 4.58e-08 & 9.96e-05 & 1.71e-09 & 32946 & 8.54e-09 & 9.41e-09 & 3.76e-10\\
4 & $10^3$ & 249880 & 365058 & 6.85e-09 & 9.92e-05 & 2.56e-10 & 33232 & 5.83e-10 & 2.88e-09 & 2.95e-11\\
5 & $10^4$ & 816213 & 1192386 & 5.24e-10 & 9.99e-05 & 2.91e-11 & 33402 & 3.42e-11 & 6.35e-09 & 2.71e-12\\\hline\hline
\multicolumn{2}{|c}{trial 2} & \multicolumn{5}{c}{ total running time = 5915 sec.} & \multicolumn{4}{c|}{total running time = 722 sec.}\\\hline
1 & 1 & 7072 & 10342 & 7.50e-02 & 9.99e-05 & 2.61e-03 & 32784 & 7.50e-02 & 9.61e-09 & 2.61e-03\\
2 & 10 & 23195 & 33900 & 8.41e-07 & 9.93e-05 & 2.99e-08 & 32950 & 7.89e-07 & 3.69e-09 & 2.78e-08\\
3 & $10^2$ & 76206 & 111344 & 4.94e-08 & 9.92e-05 & 2.04e-09 & 33234 & 2.30e-09 & 5.52e-09 & 1.63e-10\\
4 & $10^3$ & 248830 & 363524 & 4.04e-09 & 9.98e-05 & 1.42e-10 & 33458 & 3.54e-10 & 3.70e-09 & 1.15e-11\\
5 & $10^4$ & 812789 & 1187384 & 4.97e-11 & 9.96e-05 & 9.95e-12 & 33648 & 0.00e+00 & 5.79e-09 & 2.77e-12\\\hline\hline
\multicolumn{2}{|c}{trial 3} & \multicolumn{5}{c}{ total running time = 5939 sec.} & \multicolumn{4}{c|}{total running time = 699 sec.}\\\hline
1 & 1 & 7120 & 10414 & 7.10e-02 & 9.58e-05 & 2.40e-03 & 33270 & 7.10e-02 & 3.61e-09 & 2.40e-03\\
2 & 10 & 23633 & 34540 & 1.31e-06 & 9.88e-05 & 4.00e-08 & 33464 & 7.38e-07 & 7.78e-09 & 2.54e-08\\
3 & $10^2$ & 77367 & 113040 & 1.96e-09 & 9.97e-05 & 3.59e-10 & 33684 & 4.85e-09 & 3.17e-09 & 2.22e-10\\
4 & $10^3$ & 253174 & 369870 & 3.77e-10 & 9.97e-05 & 4.98e-11 & 33868 & 3.42e-10 & 9.49e-09 & 1.32e-11\\
5 & $10^4$ & 824422 & 1204378 & 4.14e-10 & 9.99e-05 & 2.04e-11 & 34194 & 7.07e-11 & 7.66e-09 & 3.69e-12\\\hline\hline
\multicolumn{2}{|c}{trial 4} & \multicolumn{5}{c}{ total running time = 5775 sec.} & \multicolumn{4}{c|}{total running time = 733 sec.}\\\hline
1 & 1 & 7012 & 10256 & 8.19e-02 & 9.26e-05 & 3.06e-03 & 32678 & 8.19e-02 & 4.81e-09 & 3.06e-03\\
2 & 10 & 23154 & 33840 & 1.25e-06 & 9.98e-05 & 4.71e-08 & 33126 & 8.47e-07 & 9.51e-09 & 3.20e-08\\
3 & $10^2$ & 76076 & 111154 & 3.45e-08 & 9.93e-05 & 1.31e-09 & 33318 & 6.70e-09 & 6.85e-09 & 2.47e-10\\
4 & $10^3$ & 247554 & 361660 & 3.82e-09 & 9.97e-05 & 1.54e-10 & 33538 & 5.65e-10 & 9.78e-09 & 2.30e-11\\
5 & $10^4$ & 803441 & 1173728 & 2.17e-10 & 1.00e-04 & 1.28e-11 & 33748 & 4.82e-11 & 3.95e-09 & 3.17e-12\\\hline\hline
\multicolumn{2}{|c}{trial 5} & \multicolumn{5}{c}{ total running time = 5887 sec.} & \multicolumn{4}{c|}{total running time = 727 sec.}\\\hline
1 & 1 & 7068 & 10338 & 7.59e-02 & 9.60e-05 & 2.67e-03 & 32528 & 7.59e-02 & 8.29e-10 & 2.67e-03\\
2 & 10 & 23384 & 34176 & 1.02e-06 & 9.85e-05 & 4.03e-08 & 32526 & 7.81e-07 & 7.20e-09 & 2.73e-08\\
3 & $10^2$ & 76462 & 111718 & 9.06e-08 & 9.97e-05 & 3.73e-09 & 32764 & 1.58e-09 & 8.64e-09 & 3.04e-10\\
4 & $10^3$ & 250963 & 366640 & 1.15e-09 & 9.99e-05 & 7.37e-11 & 32988 & 3.86e-10 & 3.95e-09 & 1.34e-11\\
5 & $10^4$ & 814436 & 1189790 & 2.02e-10 & 1.00e-04 & 7.99e-12 & 33246 & 4.75e-11 & 5.78e-09 & 2.37e-12\\\hline
\end{tabular}
}
\end{center}
\end{table}

\section{Concluding remarks}\label{sec:conclusion}
We have proposed a cutting-plane based first-order method (FOM) for solving strongly-convex problems with $m$ functional constraints. If $m=O(1)$, our method can achieve a complexity result of $\tilde O(\sqrt{\kappa})$, where $\kappa$ denotes the condition number of the underlying problem in some sense. In general, a complexity result of $\tilde O(m^2\sqrt\kappa)$ has been established. To give an $\vareps$-KKT point, our result is better than an existing lower bound if $m=o(\vareps^{-\frac{1}{4}})$. Our result can be further improved to $\tilde O(m\sqrt\kappa)$ by using a more advanced cutting-plane method as the key ingredient in our algorithm. We have also extended the idea of the cutting-plane based FOM to convex cases and nonconvex cases. Similarly, when $m=O(1)$, we obtained almost the same-order complexity results (with a difference of a polynomial of $|\log\vareps|$) as for solving an unconstrained problem.


\bibliographystyle{abbrv}
\bibliography{optim}

\end{document}